\documentclass[12pt]{amsart}
\usepackage{amsmath}
\usepackage{amssymb}
\usepackage{hyperref}
\usepackage[all]{xy}
\SelectTips{eu}{}
\usepackage{fullpage}
\usepackage[mathscr]{eucal}
\usepackage{mathtools}
\usepackage{dsfont}
\usepackage{verbatim} 

\allowdisplaybreaks
\numberwithin{equation}{section}
\numberwithin{figure}{section}
\newtheorem{lemma}[equation]{Lemma}
\newtheorem{theorem}[equation]{Theorem}
\newtheorem{prop}[equation]{Proposition}
\newtheorem{cor}[equation]{Corollary}
\newtheorem{conj}[equation]{Conjecture}

\theoremstyle{definition}
\newtheorem{question}[equation]{Question}
\newtheorem{defn}[equation]{Definition}
\newtheorem{eg}[equation]{Example}
\newtheorem{rk}[equation]{Remark}
\newtheorem{rks}[equation]{Remarks}
\newtheorem{warning}[equation]{Warning}

\renewcommand{\ge}{\geqslant}
\renewcommand{\le}{\leqslant}
\newcommand{\Ahat}{\hat A_1}
\newcommand{\bul}{{\scriptstyle\bullet}}
\newcommand{\core}{\mathsf{core}}
\newcommand{\cc}{\mathsf{c}}
\newcommand{\ep}{\varepsilon}
\newcommand{\Ker}{\mathsf{Ker}}
\newcommand{\im}{\mathsf{Im}}
\newcommand{\npj}{\gamma}
\newcommand{\one}{\mathds{1}}
\newcommand{\proj}{\mathsf{(proj)}}
\newcommand{\Rad}{\mathsf{Rad}}
\newcommand{\Soc}{\mathsf{Soc}}

\newcommand{\St}{\mathsf{St}}
\newcommand{\op}{\mathsf{op}}

\newcommand{\bC}{\mathbb{C}}
\newcommand{\bF}{\mathbb{F}}
\newcommand{\bR}{\mathbb{R}}
\newcommand{\bZ}{\mathbb{Z}}

\title{The non-projective part of the tensor powers of a module}

\author{Dave Benson}
\address{Institute of Mathematics, University of Aberdeen, Aberdeen
  AB24 3UE, Scotland, United Kingdom}

\author{Peter Symonds}
\address{School of Mathematics, University of Manchester,
Oxford Road, Manchester M13 9PL, United Kingdom}



\begin{document}

\begin{abstract}
Let $M$ be a finite dimensional modular representation of
a finite group $G$. We consider the dimensions of the
non-projective part of the tensor powers $M^{\otimes n}$ of $M$, and we
write $\npj_G(M)$ for the $\limsup$ of their $n$th roots. We investigate the properties of the
invariant $\npj_G(M)$, using tools from representation theory,
and from the theory of commutative Banach algebras.
\end{abstract}

\maketitle

\section{Introduction}

This paper is exploratory in nature. We introduce 
a new invariant $\npj_G(M)$ of a $kG$-module $M$ (where, 
throughout the paper, $G$ is a
finite group, $k$ is a field of characteristic $p$, and
we only consider finitely generated $kG$-modules). 
This invariant is difficult to 
compute, but captures interesting asymptotic properties of
tensor products. We have carried out a large number of
computations, and we present a selection of them near
the end of the paper.

The invariant $\npj_G(M)$ is similar in nature to the \emph{complexity}
$c_G(M)$ \cite{Alperin/Evens:1981a,Carlson:1981a,Kroll:1984a}, 
which describes the polynomial rate of
growth of a minimal resolution of $M$, or equivalently the
rate of growth of the non-projective part of $M \otimes
\Omega^n(k)$ (here and for the remainder of the paper, $\otimes$
means $\otimes_k$ with diagonal group action). 
Instead, we examine the rate of growth of
the dimension of the non-projective part of $M^{\otimes n}$. It turns out
that in this case the growth is usually 
exponential, so the appropriate
way to measure it is to consider $\limsup$ of the $n$th root.

\begin{defn}
For a $kG$-module $M$, we write $M=M' \oplus \proj$ where $M'$
has no projective direct summands and $\proj$ denotes a projective module.
Then $M'$ is called the \emph{core} of $M$ and denoted $\core_G(M)$.
We write $\cc_n^G(M)$ for the dimension of $\core_G(M^{\otimes n})$;
it is well defined by the Krull--Schmidt Theorem. We define
\[ \npj_G(M) = \limsup_{n\to\infty}\sqrt[n]{\cc_n^G(M)}. \]
This is also equal to $1/r$, where $r$ is the radius of convergence
of the generating function
\[ f_M(t)=\sum_{n=0}^\infty \cc_n^G(M)t^n, \]
see Corollary~\ref{co:npj}.
\end{defn}

In the case where $M$ is algebraic (i.e., there are only 
finitely many isomorphism classes of indecomposable summands
of tensor powers of $M$, see
Alperin~\cite{Alperin:1976b} or \S II.5 of Feit~\cite{Feit:1982a}) 
the invariant is a root of the minimal polynomial of
$M$ in the Green ring $a(G)$. Thus $\npj_G(M)$ is closely
related to the discussion in Chapter~3 of the book of
Etingof, Gelaki, Nikshych and 
Ostrik~\cite{Etingof/Gelaki/Nikshych/Ostrik:2015a}. 
Recall that $a(G)$ is the
ring whose generators $[M]$ correspond to isomorphism
classes of $kG$-modules, and whose relations are given
by $[M\oplus N]=[M]+[N]$ and $[M\otimes N]=[M][N]$.
Additively, the ring $a(G)$ is a free abelian group on the basis
elements $[M]$ with $M$ indecomposable. In general, the
indecomposables are unclassifiable (wild representation type)
and $a(G)$ is very large and hard to describe. It usually
has nilpotent elements of the form $[M]-[N]$, but has a large semisimple
quotient \cite{Benson/Carlson:1986a,Zemanek:1971a,Zemanek:1973a}.
See Section~\ref{se:Banach} for more notations concerning
Green rings.

All modules are algebraic if and only if the Sylow $p$-subgroups
of $G$ are cyclic. For an example, if $G$ is a cyclic group of
order five, $k$ is a field of characteristic five, and $M$ is the
two dimensional indecomposable $kG$-module, then the non-projective
summands of $M^{\otimes n}$ follow a Fibonacci pattern. We have
$\cc_{2n}^G(M)\approx \tau^{2n+1}$ and $\cc_{2n+1}^G(M)\approx
2\tau^{2n+1}$, where
\[ \tau=(1+\sqrt{5})/2 = 2\cos(\pi/5)\approx 1.618034 \]
is the golden ratio, and so $\npj_G(M)=\tau$ in this case.
We calculate the value of $\npj_G(M)$ for all modules 
for a cyclic group of order $p$ in Theorem~\ref{th:Z/p}.

Most modules are not algebraic.
In the case where $M$ is not algebraic, we interpret
$\npj_G(M)$ as the spectral radius of $[M]$ as an
element of a suitable completion of the Green ring.
This brings in the theory of commutative Banach algebras,
playing the role of a sort of infinite dimensional
Perron--Frobenius theory.

In common with the complexity, the value of $\npj_G(M)$ is
determined by the restrictions of $M$ to elementary
abelian $p$-subgroups of $G$. For this reason, most of
our examples are modules for elementary abelian $p$-groups.
The examples were worked out using the computer algebra system
\textsf{Magma}~\cite{Bosma/Cannon:1996a}. Note that this paper is
written using left modules, while \textsf{Magma} uses right modules,
so the matrices have been transposed.\medskip

The following theorem
summarises the results of this paper.

\begin{theorem}
	\label{th:omni}
The invariant $\npj_G(M)$ has the following properties:
\begin{enumerate}
\item We have
  $\npj_G(M)=\displaystyle\lim_{n\to\infty}\sqrt[n]{\cc_n^G(M)}
= \inf_{n\ge 1}\sqrt[n]{\cc_n^G(M)}$.
\item We have $0\le \npj_G(M)\le \dim M$.
\item A $kG$-module $M$ is $p$-faithful 
(Definition \ref{def:p-faithful}) if and only if $\npj_G(M) < \dim M$.
\item A $kG$-module $M$ is projective if and only if $\npj_G(M)=0$. 
Otherwise $\npj_G(M)\ge 1$.
\item If $p$ divides $|G|$ then a $kG$-module $M$ is endotrivial if and only if $\npj_G(M)=1$.
\item If a $kG$-module $M$ is neither projective nor endotrivial then
  $\npj_G(M)\ge \sqrt 2$.
\item If $\npj_G(M)=\sqrt{2}$ then $M \otimes M^*\otimes M \cong M
  \oplus M \oplus \proj$.
\item We have $\npj_G(M^*)=\npj_G(M)$.
\item If $0 \to M_1 \to M_2 \to M_3 \to 0$ is a short exact sequence
of $kG$-modules
  then each $\npj_G(M_i)$ is at most the sum of the other two.
\item We have $\max\{\npj_G(M),\npj_G(N)\}\le \npj_G(M\oplus N) \le \npj_G(M) + \npj_G(N)$.
\item If $N$ is isomorphic to a direct sum of $m$ copies of $M$ then
  $\npj_G(N)=m.\npj_G(M)$.
\item We have $\npj_G(k\oplus M)=1+\npj_G(M)$.
\item We have $\npj_G(M\otimes N) \le \npj_G(M)\npj_G(N)$.
\item We have $\npj_G(M^{\otimes m}) = \npj_G(M)^m$.
\item We have $\npj_G(\Omega M) = \npj_G(M)$.
\item If $H\le G$ then we have $\npj_H(M)\le \npj_G(M)$.
\item We have $\npj_G(M) = \displaystyle\max_{E\le G}\npj_E(M)$, where the
maximum is taken over the elementary abelian $p$-subgroups $E\le G$.
\item If $M$ is a one dimensional module and $p$ divides $|G|$  then for every $n\in  \bZ$ we have $\npj_G(\Omega^n(M))=1$.
\end{enumerate}
\end{theorem}

The proofs of these can be found as follows:
(i) in Theorem~\ref{th:inf},
(ii) and (iv) in Lemma~\ref{le:dimM},
(iii) in Theorem~\ref{th:p-faithful},
(v) in Theorem~\ref{th:endotriv},
(vi) in Theorem~\ref{th:sqrt2},
(vii) in Theorem~\ref{co:MMM},
(viii) in Lemma~\ref{le:dual},
(ix) in Corollary~\ref{co:ses},
(x) in Theorem~\ref{th:seq-sum},
(xi) in Theorem~\ref{th:m.M},
(xii) in Theorem~\ref{th:k-plus-M},
(xiii) in Theorem~\ref{th:tensor},
(xiv) in Theorem~\ref{th:tensor-powers},
(xv) in Theorem~\ref{th:Omega},
(xvi) in Lemma~\ref{le:subgroup},
and (xvii) in Theorem~\ref{th:E}.
Part (xviii) follows by combining part~(xv) with Lemma~\ref{le:dimM}\,(iv).

We use the spectral theory of commutative
Banach algebras to connect the invariant $\npj_G(M)$ to the structure of the
Green ring $a(G)$. Recall from Benson and
Parker~\cite{Benson/Parker:1984a} 
that a \emph{species} of
$a(G)$ is a ring homomorphism $s\colon a(G) \to \bC$
(we avoid the Banach algebra term ``character'' for
obvious reasons).
We say that a species $s$ is \emph{core-bounded} if for all
$kG$-modules $M$ we have $|s([M])|\le\dim\core_G(M)$.
The proof of the following theorem can be found in Section~\ref{se:Banach}.

\begin{theorem}
If $M$ is a $kG$-module then $\npj_G(M)$ is the supremum of $|s([M])|$,
where $s$ runs over the core-bounded species $s\colon a(G)\to \bC$.
Furthermore, there exists a core-bounded species $s$ of $a(G)$ such
that $s([M])=\npj_G(M)$.
\end{theorem}

We formulate some conjectures about the behaviour of the invariant
$\npj_G(M)$, two of which we restate here:

\begin{conj}[Conjecture~\ref{conj:EndM}]
We have $\npj_G(M\otimes M^*) = \npj_G(M)^2$.
\end{conj}

\begin{conj}[Conjecture~\ref{conj:er}]
For all large enough values of $n$, the function $\cc_n^G(M)$ satisfies a
homogeneous linear recurrence relation with constant coefficients.
\end{conj}

The latter conjecture implies that the value of $\npj_G(M)$ is always
an algebraic integer. We do not know whether this is the case,
but we at least show in Proposition~\ref{prop:countable} that $\npj_G(M)$,
for all primes, fields, finite groups and finitely generated modules,
can only take countably many values.\medskip

\noindent
{\bf Acknowledgment.}
This material is partly based on work of the first author supported
by the National Science Foundation under Grant No.\ DMS-1440140 while
he was in residence at the Mathematical Sciences Research Institute
in Berkeley, California, during the Spring 2018 semester, and of the
second author supported by an International Academic
Fellowship from the Leverhulme Trust.
We wish to thank Burt Totaro for his helpful comments and for pointing out Proposition~\ref{prop:countable}.

\section{The Invariant $\npj_G(M)$}

We begin with some properties of tensor products.

\begin{prop}\label{pr:AC}
Let $M$ be a $kG$-module.
\begin{enumerate}
	\item
	If the dimension of $M$ is not divisible by $p$ then $M \otimes M^*$ has a direct summand isomorphic to $k$.
\item
$M$ is isomorphic to a direct summand of
$M\otimes M^*\otimes M$.
\item
If the dimension of
$M$ is divisible by
$p$ then $M \otimes M^* \otimes M$
has a direct summand isomorphic to $M \oplus M$.
\end{enumerate}
\end{prop}
\begin{proof}
Let $m_i$ be a basis for $M$ and $f_i$ the dual basis of $M^*$.
Thus $\sum_i f_i(m_i)=\dim(M)$, and
for $m\in M$ we have $m=\sum_i f_i(m)m_i$.

For (i) we have maps $k \to M \otimes M^*$ given
by $1 \mapsto \sum_i m_i \otimes f_i$ and
$M \otimes M^*  \to M$ given by $m \otimes f 
\mapsto f(m)$, with composite multiplication by $\dim (M)$.

For (ii) we have maps $M \to M \otimes M^* \otimes M$ given
by $m \mapsto \sum_i m \otimes f_i \otimes m_i$ and
$M \otimes M^* \otimes M \to M$ given by $m \otimes f \otimes m'
\mapsto f(m)m'$, with composite the identity on $M$.

For (iii)
(cf.~Proposition 4.9 in Auslander and Carlson
\cite{Auslander/Carlson:1986a}, where this is proved
with the further hypothesis that $M$ is indecomposable, but
this hypothesis is not used in the proof), we have maps $M \oplus M \to M \otimes M^* \otimes M$
given by
\[ (m,m') \mapsto \sum_i(m \otimes f_i \otimes m_i +
m_i \otimes f_i \otimes m') \]
and $M \otimes M^* \otimes M
\to M\oplus M$ given by $m \otimes f \otimes m' \mapsto (f(m)m',f(m')m)$.
If $M$ has dimension divisible by $p$ then the composite is
the identity on $M \oplus M$.
\end{proof}

\begin{lemma}\label{le:MM*}
If $M$ is a $kG$-module, then $M \otimes M^*$ is projective if and only if $M$ is projective.
\end{lemma}
\begin{proof}
The tensor product of any module with a projective module is
projective. So by Proposition~\ref{pr:AC}, $M$ projective
implies $M\otimes M^*$ projective implies $M\otimes M^* \otimes M$
projective implies $M$ projective.
\end{proof}

\begin{lemma}\label{le:powerproj}
For any $kG$-module $M$, if $M^{\otimes n}$ is projective for some $n \ge 1$, then so is $M$.
\end{lemma}
\begin{proof}
It follows from Proposition~\ref{pr:AC} that $M^{\otimes (n-1)}$ is
isomorphic to a direct summand of $M^{\otimes n} \otimes M^*$.
So $M^{\otimes n}$ is projective if and only if $M^{\otimes (n-1)}$ is
projective. The result now follows by induction on $n$.
\end{proof}

For a $kG$-module $M$, recall that we write $\cc_n^G(M)$ for
the dimension of $\core_G(M^{\otimes n})$.

\begin{defn}\label{def:npj}
We define
\[ \npj_G(M) = \limsup_{n\to\infty}\sqrt[n]{\cc_n^G(M)}. \]
\end{defn}

\begin{rks}
\begin{enumerate}
\item
The invariant $\npj_G(M)$ is robust, in that the dimension of
$\core_G(M)$ may be replaced by the number of composition
factors of $\core_G(M)$ or the number of composition factors
of the socle of $\core_G(M)$, and so on.
\item
An interesting invariant is $\npj_G(M)/\dim M$, which we
think of as the ``non-projective proportion of $M$ in the limit.''
In the example of the introduction, we have $\npj_G(M)=\tau$.
Thus $\npj_G(M)/\dim M\approx 0.809$, and so we
think of $M$ as ``about $19.1\%$ projective in the limit.''
\item
We shall see in Section~\ref{se:submult}, using the theory
of submultiplicative functions, that in fact
$\displaystyle\lim_{n\to\infty}\sqrt[n]{\cc_n^G(M)}$ exists and is equal to
$\displaystyle\inf_{n\ge 1}\sqrt[n]{\cc_n^G(M)}$.
\end{enumerate}
\end{rks}

We begin with some obvious properties of the invariant $\npj_G(M)$.

\begin{lemma}\label{le:dimM}
For any $kG$-module we have:
\begin{enumerate}
\item
$0 \le \npj_G(M) \le \dim M$,
\item
$\npj_G(M) = 0$ if and only if $M$ is projective,
\item
If $M$ is not projective then $\npj_G(M) \ge 1$.
\item 
If $M$ is a one dimensional module and $|G|$ is divisible
by $p$ then $\npj_G(M)=1$.
\end{enumerate}
\end{lemma}
\begin{proof}
Part (i) is because
\[ \cc_n^G(M) = \dim \core_G(M^{\otimes n}) \le \dim M^{\otimes n} =
  (\dim M)^n. \]
If $M$ is projective then clearly $\npj_G(M)=0$. Conversely,
if $M$ is not projective then, by Lemma~\ref{le:powerproj},
no $\cc_n^G(M)$ is $0$. Since $\cc_n^G(M)$ is a non-negative integer we
have $\cc_n^G(M) \ge 1$, proving parts (ii) and (iii). Part (iv)
now follows, since $M$ is not projective in this case.
\end{proof}
		
\begin{lemma}\label{le:dual}
We have $\npj_G(M^*)=\npj_G(M)$.
\end{lemma}
\begin{proof}
We have $\core_G(M^*)\cong\core_G(M)^*$ and so $\cc_n^G(M^*)=\cc_n^G(M)$.
\end{proof}

Recall that we have the syzygy operator $\Omega$, where
$\Omega M$ is defined to be the kernel of a projective
cover $P \rightarrow M$. Similarly, $\Omega^{-1}M$ is
defined to be the cokernel of an injective hull $M \to I$.
Since projective $kG$-modules
are the same as injective modules,
we have $\Omega(\Omega^{-1}M)\cong \core_G(M)
\cong \Omega^{-1}(\Omega M)$.

For $n>0$, $\Omega^n M$
denotes $\Omega(\Omega^{n-1}M)$, $\Omega^{-n} M$ denotes
$\Omega^{-1}(\Omega^{-n+1}M)$, and $\Omega^0 M$ denotes $\core_G(M)$.
For $n\in\bZ$ we have
$\core_G(\Omega^n k \otimes M)\cong \Omega^n M $.

\begin{lemma}\label{le:Omega-k}
We have $\npj_G(\Omega k) = \npj_G(\Omega^{-1} k)= 1$, provided that $p$ divides $|G|$.
\end{lemma}

\begin{proof}
We have $\core_G((\Omega k)^{\otimes n}) \cong \Omega^nk$,
and $\dim \Omega^n k$ grows polynomially in $n$ (see for example \cite{Benson:1991b}\,  \S 5.3).
Therefore $\npj_G(\Omega k)=1$. Since $(\Omega k)^* \cong \Omega^{-1} k$,
Lemma \ref{le:dual} shows that $\npj_G(\Omega^{-1} k)=1$.
\end{proof}

\begin{lemma}\label{le:core}
If $\core_G(M)\cong\core_G(N)$ then $\npj_G(M)=\npj_G(N)$.
\end{lemma}
\begin{proof}
If $\core_G(M) \cong \core_G(N)$ then, since the
tensor product of a projective module with any module 
is projective, we have
\[ \core_G(M^{\otimes n})\cong\core_G(\core_G(M)^{\otimes n})
\cong \core_G(\core_G(N)^{\otimes n})\cong \core_G(N^{\otimes n}). \]
Thus $\cc_n^G(M)=\cc_n^G(N)$, and so $\npj_G(M)=\npj_G(N)$.
\end{proof}

\begin{lemma}\label{le:subgroup}
If $H$ is a subgroup of $G$ and $M$ is a $kG$-module then
$\npj_H(M)\le \npj_G(M)$.
\end{lemma}
\begin{proof}
We have $\cc_n^H(M) \le \cc_n^G(M)$ for all $n$.
\end{proof}

\begin{lemma}\label{le:K}
If $K$ is an extension field of $k$ then $\npj_G(K\otimes_k M)=\npj_G(M)$.
\end{lemma}
\begin{proof}
This will follow immediately if we can show that
for any $kG$-module $N$, we have $K \otimes_k \core_G(N)
\cong \core_G(K\otimes_k N)$. For this,
we need to show that if $K \otimes_kN$ has a projective
summand $P$ then $N$ also has a projective summand.
Consider the restrictions of $K \otimes_kN$ and $P$
from $KG$ to $kG$. The restriction $P{\downarrow_{kG}}$
must be a sum of finite dimensional indecomposable
projective $kG$-modules; let $P'$ be one of them. It is a
summand of $K \otimes_k N{\downarrow_{kG}}$, which is
a sum of copies of $N$. Because it is finite dimensional,
$P'$ is a summand of a finite sum of copies of $N$, hence
is a summand of $N$, by the Krull--Schmidt Theorem.
\end{proof}

\begin{eg}\label{eg:3x3a}
Let $M$ be the $3$ dimensional faithful uniserial module for
$G=\bZ/3\times \bZ/3=\langle g,h\rangle$ over $\bF_3$ given by
\[ g \mapsto \begin{pmatrix}1&1&0 \\ 0&1&1 \\ 0&0&1 \end{pmatrix} \qquad
h \mapsto \begin{pmatrix}1&0&1 \\ 0&1&0 \\ 0&0&1 \end{pmatrix}. \]
Then $\Omega^2M \cong M$, $M$ is algebraic
(Craven~\cite{Craven:2007a}, Section 3.3.2),
and $\Omega M$ has dimension~$6$.
The modules $M'=M\otimes M^*$ and
$\Omega M'\cong \Omega M\otimes M^*\cong M \otimes \Omega(M^*)$
are indecomposable, of dimensions $9$ and $18$ respectively.
The indecomposable summands of tensor powers of $M$
are determined by the equations
\begin{align*}
M\otimes M &\cong M^*\oplus \Omega(M^*), &
M \otimes M' &\cong 2M \oplus 2\Omega M \oplus P, \\
M \otimes \Omega M &\cong M^* \oplus \Omega(M^*) \oplus P, &
M \otimes \Omega M' &\cong 2M \oplus 2\Omega M \oplus 4P,
\end{align*}
where $P$ is the $9$ dimensional projective module. These equations
imply that
\[ M^{\otimes 5} \cong 8M^{\otimes 2} \oplus 19P. \]
It follows that for $n\ge 5$ we have $\cc_n(M)=8\cc_{n-3}(M)$, and so
$\npj_G(M)=2$. Similarly we have $\npj_G(M')=4$.
\end{eg}

\section{Short Exact Sequences and Direct Sums}

\begin{lemma}\label{le:binom}
Let $a_n$, $b_n$ and $c_n$ be sequences of non-negative real numbers, satisfying
\[ c_n \le \sum_{i=0}^n \binom{n}{i} a_i b_{n-i}. \]
Then
\[ \limsup_{n\to\infty}\sqrt[n]{c_n}\le
\limsup_{n\to\infty}\sqrt[n]{a_n}+\limsup_{n\to\infty}\sqrt[n]{b_n}. \]
\end{lemma}
\begin{proof}
The statement that $\displaystyle\limsup_{n\to\infty}\sqrt[n]{a_n}= \alpha$ implies
that for all $\ep>0$, there exists $m$ such that
for all $n\ge m$ we have $a_n \le (\alpha+\ep)^n$. Introducing
a positive constant $A$, we can assume that $a_n \le A(\alpha+\ep)^n$
for all $n\ge 0$.
Similarly, if
$\displaystyle\limsup_{n\to \infty}\sqrt[n]{b_n}=\beta$ then for all $\ep>0$ there
exists a positive constant $B$ such that for all $n\ge 0$ we have $b_n\le B(\beta+\ep)^n$.
Thus for all $\ep>0$ there is a positive constant $C=AB$ such that for all
 $n\ge 0$ we have
\begin{equation*}
c_n \le \sum_{i=0}^n \binom{n}{i}A(\alpha+\ep)^iB(\beta+\ep)^{n-i}
= C(\alpha+\beta+2\ep)^n,
\end{equation*}
and so $\displaystyle\limsup_{n\to\infty}\sqrt[n]c_n \le \alpha+\beta$.
\end{proof}

\begin{theorem}\label{th:seq-sum}
If $0 \to M_1 \to M_2 \to M_3 \to 0$ is a short exact sequence of
$kG$-modules then
\[ \npj_G(M_2) \le \npj_G(M_1) + \npj_G(M_3). \]
If the sequence splits then we also have
\[ \max\{\npj_G(M_1),\npj_G(M_3)\} \le \npj_G(M_2). \]
\end{theorem}

\begin{proof}
The module $M_2^{\otimes n}$ has a filtration of length $2^n$ where the filtered
quotients are $\binom{n}{i}$ copies of $M_1^{\otimes i} \otimes
M_3^{\otimes(n-i)}$ ($0\le i\le n$). Projective summands of a
filtered quotient split off the entire module, since they are also
injective. So
\begin{align*}
\cc^G_n(M_2)=
\dim\core_G(M_2^{\otimes n}) &\le \sum_{i=0}^n\binom{n}{i}
\dim\core_G(M_1^{\otimes i} \otimes M_3^{\otimes (n-i)}) \\
&\le \sum_{i=0}^n\binom{n}{i}\cc^G_i(M_1)\cc^G_{n-i}(M_3)
\end{align*}
Applying Lemma \ref{le:binom}, we deduce that
\[ \npj_G(M_2) \le \npj_G(M_1) + \npj_G(M_3). \]
If the sequence splits, then each $\cc^G_n(M_2)$ is at least
as big as $\cc^G_n(M_1)$ and also at least as big as $\cc^G_n(M_3)$.
\end{proof}

\begin{eg}\label{eg:SL24}
Let $M$ be the two dimensional natural module for $SL(2,\bF_4)$ and let $N$
be its Frobenius twist. Then
 $M \otimes N$ is the four dimensional Steinberg
module $\St$, which is projective. Furthermore, $M^{\otimes 3} \cong M \oplus M \oplus \St$
and $N^{\otimes 3} \cong N \oplus N \oplus \St$.
Since $M\otimes N$ is projective, for all $n\ge 1$ we have
\[ \core((M\oplus N)^{\otimes n})\cong\core(M^{\otimes n}) \oplus \core(N^{\otimes n}), \]
and so
\[ \npj_G(M)=\npj_G(N)=\npj_G(M\oplus N) = \sqrt{2}. \]
This shows that the first inequality in the theorem is not always
an equality, even for direct sums.

We shall make further use of this example in Remark \ref{rk:sqrt2}.
\end{eg}

On the other hand, for sums of isomorphic modules, we have the following.

\begin{theorem}\label{th:m.M}
If $N$ is isomorphic to a direct sum of $m$ copies of a $kG$-module $M$ then
$\npj_G(N)=m \npj_G(M)$.
\end{theorem}
\begin{proof}
The module $N^{\otimes n}$ is isomorphic to a direct sum of $m^n$ copies of $M^{\otimes n}$,
so we have $\cc_n^G(N) = m^n \cc_n^G(M)$ and $\sqrt[n]{\cc_n^G(M)}=m (\sqrt[n]{\cc_n^G(M)})$.
Now take $\displaystyle\limsup_{n\to\infty}$.
\end{proof}

Here is another useful bound.

\begin{theorem}\label{th:ge-m}
If $M_1\otimes \dots \otimes M_m$ is not projective then
$\npj_G(M_1\oplus \dots\oplus M_m) \ge m$.
\end{theorem}
\begin{proof}
By Lemma~\ref{le:powerproj}, no power of
$M_1\otimes \dots \otimes M_m$ is projective, so
neither is any module of the form $M_1^{i_1}\otimes \dots \otimes M_m^{i_m}$.
The module $(M_1\oplus \dots \oplus M_m)^{\otimes n}$ thus has at least $m^n$
non-projective summands, so we have
\[ \cc_n^G(M_1\oplus \dots \oplus M_m) \ge m^n \]
and
\[ \sqrt[n]{\cc_n^G(M_1\oplus \dots \oplus M_m)} \ge m. \]
Now take $\displaystyle\limsup_{n\to \infty}$.
\end{proof}

\begin{cor}\label{co:k-plus-M}
If $M$ is not projective then $\npj_G(k \oplus M) \ge 2$.
\end{cor}
\begin{proof}
This follows by taking $m=2$, $M_1=k$ and $M_2=M$ in Theorem \ref{th:ge-m}.
\end{proof}

\begin{rk}
We shall prove in the next section,
using the theory of submultiplicative sequences, that
$\npj_G(k \oplus M)$ is always equal to $1+\npj_G(M)$.
\end{rk}

\section{Submultiplicative Sequences}\label{se:submult}

In this section, we investigate the submultiplicative properties of
$\npj_G$, and deduce Theorem \ref{th:k-plus-M}.
We shall revisit this from the point of view of Banach
algebras and Gelfand's spectral radius theorem later on,
but for the moment we shall try to stay elementary.

\begin{defn}
We say that a sequence $c_0,c_1,c_2,\dots$ of non-negative real
numbers is \emph{submultiplicative} if $c_0=1$, and for all $m,n\ge 0$ we have
$c_{m+n} \le c_m.c_n$.
\end{defn}

\begin{lemma}\label{le:ccMsubmult}
If $M$ is a $kG$-module then $\cc_n(M)$ is a submultiplicative
sequence, provided that $p$ divides $|G|$.
\end{lemma}
\begin{proof}
This follows from the fact that
\begin{equation*}
\core_G(M^{\otimes m}\otimes M^{\otimes n})\cong
\core_G(\core_G(M^{\otimes m}) \otimes \core_G(M^{\otimes n})).
\qedhere
\end{equation*}
\end{proof}

\begin{lemma}[Fekete \cite{Fekete:1923a}]\label{le:Fekete}
If $c_n$ is a submultiplicative sequence then
\[ \limsup_{n\to\infty} \sqrt[n]{c_n} = \lim_{n\to\infty}\sqrt[n]{c_n}= \inf_{n\ge 1}
  \sqrt[n]{c_n}. \]
\end{lemma}
\begin{proof}
It suffices to show that $\limsup_{n\to\infty} \sqrt[n]{c_n} \leq \inf_{n\ge 1}
\sqrt[n]{c_n}$. If some $c_n$ is equal to zero, then so are all subsequent ones. So we
assume that all $c_n>0$. Suppose that $L$ is a number such that
\[ \inf_{n\to\infty}\sqrt[n]{c_n} < L. \]
Then there is an $m\ge 1$ with $\sqrt[m]{c_m}<L$. For $n>m$ we use  division
with remainder to write $n=mq_m+r_m$ with $0\le r_m<m$.
By the definition of submultiplicativity, we have
\[ c_n = c_{mq_m+r_m}\le c_{mq_m}c_{r_m}\le (c_{m})^{q_m}c_{r_m}. \]
Now $q_m\le n/m$, so $q_m/n\le 1/m$. So we have
\[ \sqrt[n]{c_n} \le \sqrt[m]{c_m}\sqrt[n]{c_{r_m}} <
  L.\sqrt[n]{c_{r_m}}. \]
As $n$ tends to infinity, the numbers
$\sqrt[n]{c_0},\dots,\sqrt[n]{c_{m-1}}$ all tend to one, and so
\begin{equation*}
\limsup_{n\to\infty}\sqrt[n]{c_n}\le L.
\qedhere
\end{equation*}
\end{proof}

\begin{theorem}\label{th:inf}
If $M$ is a $kG$-module then
$\npj_G(M)=\displaystyle\lim_{n\to\infty}\sqrt[n]{\cc_n^G(M)}
=\inf_{n\ge 1}\sqrt[n]{\cc_n^G(M)}$.
\end{theorem}
\begin{proof}
This follows from Lemmas \ref{le:ccMsubmult} and \ref{le:Fekete}.
\end{proof}

\begin{prop}\label{pr:binom2}
Suppose that $a_n$ and $b_n$ are submultiplicative sequences. Define
a sequence $c_n$ by
\[ c_n = \sum_{i=0}^n\binom{n}{i}a_ib_{n-i}. \]
Then $c_n$ is also a submultiplicative sequence, and we have
\[  \lim_{n\to\infty}\sqrt[n]{c_n}=
\lim_{n\to\infty}\sqrt[n]{a_n} +\lim_{n\to\infty}\sqrt[n]{b_n}. \]
\end{prop}
\begin{proof}
Using the fact that
\[ \binom{m+n}{\ell} =\sum_{i+j=\ell}\binom{m}{i}\binom{n}{j}\]
and the submultiplicativity of the sequences $a_n$ and $b_n$, we have
\[ \sum_{\ell=0}^{m+n}\binom{m+n}{\ell}a_\ell b_{m+n-\ell} \le
\left(\sum_{i=0}^m\binom{m}{i}a_ib_{m-i}\right).
\left(\sum_{j=0}^n\binom{n}{j}a_jb_{n-j}\right) \]
and so the sequence $c_n$ is submultiplicative.

By Lemma \ref{le:binom} we have
\[ \lim_{n\to\infty}\sqrt[n]{c_n}\le
\lim_{n\to\infty}\sqrt[n]{a_n} +\lim_{n\to\infty}\sqrt[n]{b_n}. \]
The reverse inequality is proved similarly. If
$\displaystyle\lim_{n\to\infty}\sqrt[n]{a_n}=\alpha$ and
$\displaystyle\lim_{n\to\infty}\sqrt[n]{b_n}=\beta$ then given $\ep>0$ there
exist positive constants $A$ and $B$ such that for all $n\ge 0$ we have
$a_n\ge A(\alpha-\ep)^n$ and $b_n\ge B(\beta-\ep)^n$. So for
all $\ep>0$ there is a positive constant $C=AB$ such that for all
$n\ge 0$ we  have
\begin{equation*}
c_n \ge \sum_{i=0}^n \binom{n}{i}A(\alpha-\ep)^iB(\beta-\ep)^{n-i}
= C(\alpha+\beta-2\ep)^n,
\end{equation*}
and so
\begin{equation*}
\lim_{n\to\infty}\sqrt[n]{c_n}\ge\alpha+\beta-2\ep=
\lim_{n\to\infty}\sqrt[n]{a_n} +\lim_{n\to\infty}\sqrt[n]{b_n}-2\ep.
\qedhere
\end{equation*}
\end{proof}

\begin{theorem}\label{th:k-plus-M}
If $p$ divides $|G|$ and $M$ is a $kG$-module then we have $\npj_G(k\oplus M)=1+\npj_G(M)$.
\end{theorem}
\begin{proof}
We have
\[ \cc_n^G(k\oplus M)=\sum_{i=0}^n\binom{n}{i}\cc_i^G(M). \]
So we can apply Proposition~\ref{pr:binom2} with $a_n=\cc_n^G(M)$,
$b_n=1$, and $c_n=\cc_n^G(k\oplus M)$.
\end{proof}

\begin{cor}\label{co:ab}
If If $p$ divides $|G|$, $M$ is a $kG$-module, and $N$ is isomorphic to a direct sum of
$a$ copies of $k$ and $b$ copies of $M$ then we have
$\npj_G(N)=a+b\npj_G(M)$.
\end{cor}
\begin{proof}
This follows inductively from Theorems \ref{th:m.M} and \ref{th:k-plus-M}.
\end{proof}

\section{Tensor Products}

\begin{theorem}\label{th:tensor}
We have $\npj_G(M\otimes N) \le \npj_G(M)\npj_G(N)$.
\end{theorem}
\begin{proof}
We have
\[ \core_G(M\otimes N) = \core_G(\core_G(M) \otimes \core_G(N)). \]
Therefore
\[ \cc_n^G(M\otimes N) \le \cc_n^G(M)\cc_n^G(N) \]
and
\[ \sqrt[n]{\cc_n^G(M \otimes N)} \le \sqrt[n]{\cc_n^G(M)}\sqrt[n]{\cc_n^G(N)}. \]
Now apply $\displaystyle\limsup_{n\to\infty}$ to both sides.
\end{proof}

This inequality may be strict. For example, it is possible for $M\otimes N$ to
be projective with neither $M$ nor $N$ projective. However for tensor powers
of a single module, we have the following.

\begin{theorem}\label{th:tensor-powers}
We have $\npj_G(M^{\otimes m}) = \npj_G(M)^m$.
\end{theorem}
\begin{proof}
By Theorem \ref{th:tensor} we have $\npj_G(M^{\otimes m})\le \npj_G(M)^m$.
Conversely,
if $n=ms+i$ with $0\le i< m$ then
\[ M^{\otimes n} =M^{\otimes i} \otimes (M^{\otimes m})^{\otimes s} \]
and so
\[ \cc_n^G(M) \le (\dim M)^m \cc_s^G(M^{\otimes m}). \]
Thus
\begin{equation*}
\sqrt[n]{\cc_n^G(M)} \le \sqrt[ms]{\cc_n^G(M)}
\le \sqrt[s]{\dim M} \sqrt[m]{\sqrt[s]{\cc_s^G(M^{\otimes m})}}.
\end{equation*}
Applying $\displaystyle\limsup_{n\to\infty}$, the factor $\sqrt[s]{\dim M}$ tends to $1$.
It follows that
\begin{equation*}
\npj_G(M) \le \sqrt[m]{\npj_G(M^{\otimes m})}.
\qedhere
\end{equation*}
\end{proof}

The following conjecture is based on extensive computations,
but we have failed to find a proof.
See Remark~\ref{rk:Banach} for an interpretation 
in terms of Banach algebas.

\begin{conj}\label{conj:EndM}
We have $\npj_G(M\otimes M^*) = \npj_G(M)^2$.
\end{conj}

\begin{theorem}\label{th:Omega}
We have $\npj_G(\Omega M) = \npj_G(M)$.
\end{theorem}
\begin{proof}
We have $\core_G(\Omega k \otimes M) \cong\core(\Omega M)= \Omega M$.
So by Lemma \ref{le:core}, Theorem \ref{th:tensor} and Lemma \ref{le:Omega-k} we have
\[ \npj_G(\Omega M) = \npj_G(\Omega k \otimes M) \le  \npj_G(\Omega k)\npj_G(M)=\npj_G(M). \]
The reverse inequality follows in the same way from the fact that
\begin{equation*}
\core_G(\Omega^{-1}k \otimes \Omega M) \cong \core_G(M).
\qedhere
\end{equation*}
\end{proof}

\begin{eg}\label{eg:3x3b}
Let $M$ be the three dimensional module $\Soc^2(kG)$ for
$G=\bZ/3\times\bZ/3=\langle g,h\rangle$ over $\bF_3$, given by the
following matrices, which has the diagram shown:
\[ g \mapsto \begin{pmatrix}1&1&0 \\ 0&1&0 \\ 0&0&1 \end{pmatrix} \qquad
h \mapsto \begin{pmatrix}1&0&1 \\ 0&1&0 \\ 0&0&1 \end{pmatrix} \qquad
\vcenter{\xymatrix@=4mm{\bul \ar@{-}[dr]&&\bul\ar@{-}[dl]\\&\bul}} \]
In this diagram and those in Section~\ref{se:eg}, the vertices
represent basis vectors.  The actions of $g-1$ and $h-1$ 
are represented by the lines going down to the left, 
respectively down to the right from the vertex, or zero if there is no
such line. 
 Then $M$ is non-periodic and
non-algebraic, and
(Craven~\cite{Craven:2007a}, Section 3.3.2)
we have
\[ M \otimes M \cong M^* \oplus \Omega(M^*). \]
Here, $M^*\cong kG/\Rad^2(kG)$ has dimension three,
and $\Omega(M^*)\cong\Soc^3(kG)$ has dimension six.
Using this, and the fact that $M'=M\otimes M^*$ is a non-projective
indecomposable module, it is easy to compute that $\core_G(M^{\otimes n})$ has
$2^{n-2}$ non-projective summands if $n$ is divisible by three,
and $2^{n-1}$ non-projective summands otherwise. So using Theorem
\ref{th:ge-m}
we have $\cc_n^G(M)\ge 2^{n-2}$ and $\npj_G(M)\ge 2$. On the other
hand, using Theorem~\ref{th:tensor-powers},  Theorem~\ref{th:seq-sum},
Lemma~\ref{le:dual} and Theorem~\ref{th:Omega} we have
\[ \npj_G(M)^2 = \npj_G(M\otimes M) = \npj_G(M^*\oplus \Omega(M^*))
\le \npj_G(M^*)+\npj_G(\Omega(M^*)) = 2\npj_G(M) \]
and so $\npj_G(M) \le 2$. Combining these, we have $\npj_G(M)=2$.
Similarly, we have $\npj_G(M')=4$.
\end{eg}

\begin{cor}\label{co:ses}
Let $0\to M_1\to M_2\to M_3\to 0$ be a short exact sequence of
$kG$-modules. Then $\npj_G(M_i)\le \npj_G(M_j)+\npj_G(M_\ell)$,
for any $\{i,j,\ell\}=\{1,2,3\}$.
\end{cor}
\begin{proof}
This follows from Theorems~\ref{th:seq-sum}
and~\ref{th:Omega}, together with the observation
that there are short exact sequences
\begin{gather*}
0 \to M_2 \to M_3 \oplus \proj \to\Omega^{-1}(M_1)\oplus \proj \to 0 \\
0 \to \Omega(M_3)\oplus \proj \to M_1 \oplus \proj \to M_2 \to 0.
\qedhere
\end{gather*}
\end{proof}

\begin{defn}
Recall that a $kG$-module $M$ is \emph{endotrivial} if
$M \otimes M^* \cong k \oplus \proj$.
\end{defn}

\begin{theorem}\label{th:sqrt2}
If $M$ is neither projective nor endotrivial then
$\npj_G(M)\ge \sqrt 2$.
\end{theorem}
\begin{proof}
Suppose that $M$ is neither projective nor endotrivial.
We divide into two cases according to whether the dimension
of $M$ is divisible by $p$.

If the
dimension of $M$ is divisible by $p$, then
by Proposition \ref{pr:AC}, $(M\otimes M^*)^{\otimes 2}$ has a direct
summand isomorphic to a direct
sum of two copies of $M\otimes M^*$.
Thus, $\npj_G(M \otimes M^*)^2 =
\npj_G((M \otimes M^*)^{\otimes 2}) \ge
2  \npj_G (M \otimes M^*)$. As $M$ is not projective,
neither is $M \otimes M^*$ by Lemma~\ref{le:MM*} and we
have $\npj_G (M \otimes M^*) \ne 0$, so  $\npj_G(M\otimes M^*)\ge 2$.

On the other hand, if the dimension of $M$ is not divisible by $p$
 and $M$ is not endotrivial then, by Proposition~\ref{pr:AC}, $M \otimes M^*
\cong k\oplus X$ with $X$ non-projective. Then
$\npj_G(X) \ge 1$, so $\npj_G( M \otimes M^*) \ge 2$, by Corollary~\ref{co:k-plus-M}.

In both cases,
using Theorem~\ref{th:tensor} and Lemma~\ref{le:dual}, we have
\[ 2 \le \npj_G(M\otimes M^*) \le \npj_G(M)\npj_G(M^*) =  \npj_G(M)^2 \]
and so $\npj_G(M)\ge \sqrt 2$.
\end{proof}

\begin{rk}\label{rk:sqrt2}
Example \ref{eg:SL24} shows that equality can occur in
Theorem \ref{th:sqrt2}. See Section~\ref{se:sqrt2} for more
about what happens when $\npj_G(M)\ge \sqrt 2$.
\end{rk}

We can deduce a theorem of Carlson on finite dimensional idempotent
$kG$-modules (see Theorem~3.5 of \cite{Carlson:1996a}) as a corollary.

\begin{cor}
If $M\otimes M\cong M\oplus \proj$ and $M$ is not projective then $M \cong
k\oplus\proj$.
\end{cor}
\begin{proof}
The hypothesis implies that $\npj_G(M)=1$, so by Theorem~\ref{th:sqrt2},
$M$ is endotrivial. The endotrivial modules (modulo projective summands)
form a group under tensor
product, so the only idempotent element is the identity.
\end{proof}

\section{Faithful Modules}

\begin{defn}\label{def:p-faithful}
We say that a $kG$-module $M$ is \emph{$p$-faithful} if it is not 0 and
no element of order $p$ in $G$ acts trivially on $M$. So faithful
implies $p$-faithful, and $p$-faithful is equivalent to being
faithful on restriction to a Sylow $p$-subgroup of $G$.
\end{defn}

\begin{lemma}\label{le:p-faithful}
Let $M$ be a $kG$-module. Then
some tensor power $M^{\otimes n}$ with $n \ge 1$ has a non-zero projective summand
if and only if $M$ is $p$-faithful.
\end{lemma}

\begin{proof}
The lemma is clearly true if $M=0$, so assume that $M \ne 0$.
If $M$ is not $p$-faithful, then there is an element $g\in G$
of order $p$ acting trivially on $M$. It therefore acts trivially on
$M^{\otimes n}$, so this module has no projective summands.

On the other hand, if $M$ is $p$-faithful then the kernel of the
action on $M$ is a normal $p'$-subgroup $H\le G$. Projective
$kG/H$-modules are projective $kG$-modules, so we may assume that
$H=1$. This case follows from the Corollary to Theorem~1 of 
Bryant and Kov\'acs \cite{Bryant/Kovacs:1972a}.
\end{proof}

\begin{theorem}\label{th:p-faithful}
We have $\npj_G(M) < \dim M$ if and only if $M$ is $p$-faithful.
\end{theorem}

\begin{proof}
Again, we may assume that $M \ne 0$. We use Lemma \ref{le:p-faithful}.
If $M$ is not $p$-faithful then for all $n$ we have
$\core_G(M^{\otimes n})=M^{\otimes n}$ and so $\npj_G(M)=\dim M$.
Conversely, if $M$ is $p$-faithful, then some tensor power
has a projective summand, say $M^{\otimes m} = P \oplus N$
with $P$ a non-zero projective module. Thus using Theorem \ref{th:tensor-powers} we have
\[ \npj_G(M)^m=\npj_G(M^{\otimes m})\le \dim N < (\dim M)^m \]
and so $\npj_G(M)< \dim M$.
\end{proof}

\section{Restriction to Elementary Abelian Subgroups}\label{se:elemab}

\begin{theorem}\label{th:Carlson}
There exists a constant $B$, which depends only on $p$ and $G$,
such that if $M$ is a $kG$-module then
\[ \dim \core_G(M) \le B  \max_{E\le G}\dim\core_E(M) \]
where the maximum is taken over the set of elementary
abelian $p$-subgroups $E$ of $G$.
\end{theorem}
\begin{proof}
See Theorem~3.7 of Carlson~\cite{Carlson:1981c}.
\end{proof}

\begin{theorem}\label{th:E}
Let $M$ be a $kG$-module. Then
$\npj_G(M) = \displaystyle\max_{E\le G} \npj_E(M)$.
\end{theorem}
\begin{proof}
By Theorem \ref{th:Carlson} and Lemma \ref{le:subgroup} we have
\begin{equation*}
\max_{E\le G}\sqrt[n]{\cc_n^E(M)}\le
\sqrt[n]{\cc_n^G(M)} \le
\sqrt[n]{B}\,\max_{E\le G}\sqrt[n]{\cc_n^E(M)} .
\end{equation*}
Taking $\displaystyle\limsup_{n\to\infty}$, the factor of $\sqrt[n]{B}$ tends to $1$.
\end{proof}

\begin{eg}
Let $G$ be a generalised quaternion group and let $k$ be a field of
characteristic two. Then $G$ has only one elementary abelian $2$-subgroup
$E=\langle z\rangle$, where $z$ is the central element of order two.
Let $X=1+z$, an element of $kG$ satisfying $X^2=0$.
If $M$ is a $kG$-module then the restriction to $kE$ is a direct sum
of $\dim(\Ker(X,M)/\im(X,M))$ copies of the trivial module plus a free module.
It follows that
\[ \npj_G(M)=\dim(\Ker(X,M)/\im(X,M)). \]
In particular, this is an integer.
\end{eg}

\begin{prop}[Dade \cite{Dade:1978a,Dade:1978b}]\label{pr:Dade}
If $E$ is an elementary abelian $p$-group, then the only
indecomposable endo\-trivial $kE$-modules
are the syzygies $\Omega^n(k)$ ($n\in\bZ$) of the trivial module.\qed
\end{prop}

\begin{theorem}\label{th:endotriv}
If $p$ divides $|G|$ then a $kG$-module $M$ is endotrivial if and only if $\npj_G(M)=1$.
\end{theorem}
\begin{proof}
If $M$ is projective then $\npj_G(M)=0$. If $M$ is neither projective
nor endotrivial then by Theorem \ref{th:sqrt2} we have $\npj_G(M)\ge \sqrt{2}$.

Conversely if $M$ is endotrivial then its restriction to
every elementary abelian $p$-subgroup of $G$ is endotrivial.
So by Theorem \ref{th:E}, we may assume that $G=E$ is
an elementary abelian $p$-group. Since an endotrivial
module is a direct sum of an indecomposable endotrivial
module and a projective module, we may assume that
$M$ is indecomposable.
By Proposition~\ref{pr:Dade},
$M$ is a syzygy of the trivial module, so
by Theorem \ref{th:Omega} we have $\npj_E(M)=1$.
\end{proof}

\begin{warning}
If $E$ is an elementary abelian $p$-group and $M$ is a $kE$-module then
$\npj_E(M)$ does depend on the Hopf algebra structure of $kE$. If we
regard $kE$ as the universal enveloping algebra of a restricted Lie algebra with
trivial bracket and trivial $p$th power map, and we use the corresponding
comultiplication, then $\npj_E(M)$ may change. For example, restrict
the module of Example~\ref{eg:SL24} to a Sylow $2$-subgroup, which is
elementary abelian of order four. Then by Theorem~\ref{th:E},
we have $\npj_E(M)=\sqrt{2}$. But if we
use the Lie comultiplication then $M\otimes M \cong M \oplus M$,
and so $\npj_E(M)=2$.
\end{warning}

\section{Radius of Convergence}\label{se:gamma}

Another way of studying the invariant $\npj_G(M)$ is to consider power series;
we begin with a well known lemma from analysis.

\begin{lemma}[Cauchy, Hadamard]\label{le:radius}
Let $\phi\colon \bZ_{\ge 0} \to \bC$. Then the radius of
convergence $r$ of the power series
\[ f(t)=\sum_{n=0}^\infty \phi(n)t^n \]
is given by
\[ 1/r = \limsup_{n \to \infty} \sqrt[n]{|\phi(n)|}. \]
For $|t|<r$, the convergence is uniform and absolute.
\end{lemma}
\begin{proof}
See for example
Conway \cite{Conway:1973a}, Theorem~III.1.3.
\end{proof}

\begin{cor}\label{co:npj}
Let $M$ be a $kG$-module. Consider the power series
\[ f_M(t) = \sum_{n=0}^\infty \cc_n^G(M) t^n, \]
and let $r$ be the radius of convergence of $f_M(t)$. Then
\[ 1/r = \npj_G(M). \]
\end{cor}

The following theorem will be used in Section~\ref{se:recursive}.

\begin{theorem}[Pringsheim]\label{th:Pringsheim}
Suppose that $\phi\colon \bZ_{\ge 0} \to \bR_{\ge 0}$, and that the
power series
\[ f(t)=\displaystyle\sum_{n=0}^\infty \phi(n)t^n \]
has radius of
convergence $r$. Then $t=r$ is a singular point of $f(t)$.
\end{theorem}
\begin{proof}
See Statement (7.21) in Chapter VII of Titchmarsh~\cite{Titchmarsh:1939a}.
\end{proof}

\section{Banach Algebras}\label{se:Banach}

We recall the basics of the theory of norms and spectral radius,
referring to Chapters~17--18 of Lax~\cite{Lax:2002a} for proofs.
We always work over the field of complex numbers.

\begin{defn}
A \emph{normed space} is a vector space $B$ over $\bC$,  together with a norm $B \to \bR$, $x
\mapsto \|x\|$, satisfying
\[  \|x+y\|\le\|x\|+\|y\|,\qquad\|cx\|=|c|\|x\|, \]
for $x,y\in B$, $c\in\bC$, such that $\|x\| \ge 0$ and $\|x\|=0$ if and only if $x=0$.
A \emph{Banach space} is a normed space
that is complete with respect to the norm.

A (unital) \emph{normed algebra} is an associative algebra $A$ over
$\bC$ with identity~$\one$ that is also a normed space, with the norm satisfying the additional conditions
\[ \|\one\|=1, \qquad\|xy\|\le\|x\|\|y\| \]
for $x,y\in A$.
A \emph{Banach algebra} is a normed algebra
that is also a Banach space. Note that we are assuming that
all our Banach algebras are unital.

If $a$ is an element of a Banach algebra $A$, we write
$\sigma(a)$ for the \emph{spectrum} of $a$, namely the
set of $\lambda\in\bC$ such that $\lambda\one-a$ is
not invertible in $A$. It is a non-empty closed bounded subset of
$\bC$. The \emph{spectral radius} of $a\in A$, denoted $\rho(a)$, is defined to be
$\displaystyle \sup_{\lambda\in\sigma(a)}|\lambda|$.
\end{defn}

Notice that if $A$ is finite dimensional and $a\in A$, then the spectral radius
$\rho(a)$ is just the largest absolute value of an eigenvalue of the
linear map induced by multiplying by $a$.

Let $A(G)=\bC\otimes_\bZ a(G)$, where $a(G)$
is the \emph{Green ring} or \emph{representation ring} of
finitely generated $kG$-modules.
Following \cite{Benson/Parker:1984a}, we write elements of
$A(G)$ in the form $\sum_i a_i[M_i]$ where $a_i\in\bC$ and
$M_i$ are indecomposable $kG$-modules. If $M=\bigoplus_i  M_i^{n_i}$
is a $kG$-module, we write $[M]$ for $\sum_i n_i[M_i] \in A(G)$.
Multiplication is extended bilinearly from $[M][N]=[M\otimes N]$.

We write $a(G,1)$ for the ideal of $a(G)$ spanned by the elements
$[P]$ with $P$ projective, and $a_0(G,1)$ for the linear span of the
elements of the form $[M_2]-[M_1]-[M_3]$ where
$0\to M_1\to M_2\to M_3 \to 0$ is a short exact sequence.
Then defining $A(G,1)=\bC\otimes_\bZ a(G,1)$ and 
$A_0(G,1)=\bC\otimes_\bZ a_0(G,1)$, we have
\[ A(G)=A(G,1)\oplus A_0(G,1). \]

We put a norm on $A(G)/A(G,1)\cong A_0(G,1)$ by setting
\[ \left\|\sum_ia_i[M_i]\right\|=\sum_i|a_i|\dim\core_G(M_i)
= \sum_{M_i \text{ non-projective}}|a_i|\dim M_i. \]
The reason for choosing this particular norm is that it
has two good properties:
\begin{enumerate}\item If $M$ is a $kG$-module then
$\|[M]\|=\dim\core_G(M)$; thus $\| \cdot \|$ is additive
on direct sums of genuine (as opposed to virtual) modules.
\item If $H$ is a subgroup of $G$ then restriction can only
reduce the norm: $\| M{\downarrow ^G_H} \| \le \| M \|$.
\end{enumerate}
This makes $A(G)/A(G,1)$ into a normed algebra, which we may
complete with respect to the norm to obtain a commutative
Banach algebra which we shall denote $\Ahat(G)$. Thus $A(G)/A(G,1)$
is a dense subalgebra of $\Ahat(G)$.

\begin{warning}
If $p$ does not divide $|G|$ then $\Ahat(G)=0$, which is not a
Banach algebra because it does not satisfy the condition
$ \|\one\|=1$. In this paper we always implicitly assume that the
characteristic of the field $k$ divides the order of the group.
\end{warning}

The role of the invariant
$\npj_G(M)$ in this context is that by Theorem~\ref{th:inf} we have
\begin{equation}\label{eq:norm}
\npj_G(M)=\lim_{n\to\infty} \sqrt[n]{\|[M]^n\|}.
\end{equation}

\begin{prop}[Spectral radius formula, Gelfand \cite{Gelfand:1941a}]\label{pr:Gelfand}
If $A$ is a Banach algebra and $a\in A$ then the spectral radius of
$a$ is related to the norm by the formula
\begin{equation}\label{eq:Gelfand}
\rho(a)=
  \lim_{n\to\infty}\sqrt[n]{\|a^n\|}.
\end{equation}
\end{prop}
\begin{proof}
See for example \S17.1 of Lax \cite{Lax:2002a}.
\end{proof}

\begin{theorem}\label{th:spec-radius}
If $M$ is a $kG$-module then $\npj_G(M) = \rho([M])$, where $[M]$ is the corresponding element of $\Ahat(G)$.
\end{theorem}
\begin{proof}
This follows from \eqref{eq:norm} and Proposition~\ref{pr:Gelfand}.
\end{proof}

\begin{lemma}\label{le:1+a}
Let $a$ be an element in a Banach algebra
$A$, with $\rho(a)=r$. Then $r$ (as a real number) is an element of
$\sigma(a)\subseteq\bC$ if and only if $\rho( \one +
a)=1+r$.
\end{lemma}
\begin{proof}
It is clear that $\sigma(\one + a)$ is the set of $1+\lambda$ with
$\lambda\in\sigma(a)$. So $\sigma(\one + a)$ is contained in a disc of
radius $r$ centred at $1\in\bC$. The only point in this disc at
distance $1+r$ from the origin is the real number $1+r$. Now using the
fact that $\sigma(a)$ is closed, we see that the spectral radius of
$\one + a$ is $1+r$ if and only if $1+r\in \sigma(\one + a)$, namely if
and only if $r\in\sigma(a)$.
\end{proof}

\begin{theorem}\label{th:SpecM}
Let $M$ be a $kG$-module. Then
the real number $\npj_G(M)$ is an element of $\sigma([M])\subseteq
\bC$.
\end{theorem}

\begin{proof}
By Theorem~\ref{th:k-plus-M}, we have $\npj_G(k\oplus M)=1+\npj_G(M)$.
By Theorem~\ref{th:spec-radius}, it follows that on $\hat A_1(G)$ we have
 $\rho(\one + [M])= 1+\rho([M])$. By Lemma~\ref{le:1+a} this implies that $\npj_G(M)\in\sigma([M])$.
\end{proof}

The way to connect spectral radius with the species of the Green ring
in the sense of Benson and Parker \cite{Benson/Parker:1984a} is the following.

\begin{theorem}\label{th:invertible}
An element $a$ of a commutative Banach algebra $A$ is invertible if and
only if $\phi(a)\ne 0$ for all algebra homomorphisms
$\phi\colon A\to\bC$.
\end{theorem}
\begin{proof}
See Theorem~3 in Chapter~18 of Lax \cite{Lax:2002a}.
\end{proof}

\begin{rk}\label{rk:continuous}
Note that if $A$ is a commutative Banach algebra and
$\phi\colon A \to \bC$ is an
algebra homomorphism then for all $a\in A$ we have
$|\phi(a)|\le\|a\|$. It follows that $\phi$ is automatically continuous
with respect to the norm.
See Theorem~1 in Chapter~18 of Lax~\cite{Lax:2002a}.
\end{rk}

\begin{cor}\label{co:radius}
If $a$ is an element of a commutative Banach algebra then
$\sigma(a)$ is the set of values of $\phi(a)$ as $\phi$ runs over
the algebra homomorphisms $A \to \bC$.
The spectral radius $\rho(a)$
is equal to $\displaystyle\sup_{\phi\colon A \to \bC} |\phi(a)|$.
\end{cor}
\begin{proof}
It follows from Theorem~\ref{th:invertible} that
$\lambda\one-a$ is not invertible if and only if there
exists an algebra homomorphism
$\phi\colon A \to \bC$ such that $\phi(a)=\lambda$.
\end{proof}

\begin{defn}
Recall from \cite{Benson/Parker:1984a} that a \emph{species}
of $a(G)$ is a ring homomorphism $s\colon a(G)\to \bC$.
A species of $a(G)$ extends to give an algebra homomorphism $s\colon A(G) \to
\bC$, and all algebra homomorphisms have this form.

We say that a species $s$ of $a(G)$ is \emph{core-bounded} if for all
$kG$-modules $M$ we have
\[ |s([M])|\le\dim\core_G(M). \]
In particular, the extension of a core-bounded species to $A(G)$
vanishes on $A(G,1)$, and so
defines an algebra homomorphism $A(G)/A(G,1)\to\bC$.
So for example the Brauer species, namely the ones that vanish on
$A_0(G,1)$, are not core-bounded because they
do not vanish on projective modules.
\end{defn}

\begin{lemma}
\label{le:s-Omega}
If $s$ is any core-bounded species then $|s([\Omega^i(k)])|=1$ for any $i \in \bZ$.
\end{lemma}
\begin{proof}
We have $|s([\Omega k])|^i = |s([\Omega k]^i)|= |s([\Omega^i k])| \le
\dim \core_G (\Omega^i k)$. But $\dim \core_G (\Omega^i k)$ grows
polynomially in $i$ (see for example \cite{Benson:1991b}\,\S 5.3), so
$|s([\Omega k])| \le 1$. The same holds for $|s([\Omega^{-1} k])|$;
but $|s([\Omega k])| \, |s([\Omega^{-1} k])|=|s([k])|=1$, so we
must have $|s([\Omega k])| = |s([\Omega^{-1} k])| =1$.
The general case follows from the first formula in this proof.
\end{proof}

\begin{eg}
Examining the species for $\bZ/2\times\bZ/2$
described in Appendix~1 of~\cite{Benson/Parker:1984a}, we see that
not every species that vanishes on $A(G,1)$ is core-bounded.
In this example, the quotient of $A(G)$ by the ideal spanned
by the indecomposables of even dimension is isomorphic to
the group algebra $\bC[t,t^{-1}]$, via an isomorphism sending
$\Omega(k)$ to $t$ and $\Omega^{-1}(k)$ to $t^{-1}$.
So there are species $s_z$ parametrised by the non-zero $z\in\bC$,
which factor through this quotient, and satisfy 
$s_z(\Omega(k))=z$, $s_z(\Omega^{-1}(k))=z^{-1}$. 
Only the ones with $z$
on the unit circle are core-bounded. This is because the dimension
of $\Omega^n(k)$ is $2|n|+1$, whereas if $z$ is not on the
unit circle then either the powers $z^n$ or the powers $z^{-n}$
grow exponentially with $n$ in absolute value.
\end{eg}

The following proposition shows that there is a natural correspondence
between core-bounded species of $a(G)$ and algebra homomorphisms
$\Ahat(G) \to \bC$.

\begin{prop}\label{pr:core-bounded}
For a species $s\colon A(G)\to\bC$, the following are equivalent:
\begin{enumerate}
\item $s$ is core-bounded.
\item For all $x\in A(G)$ we have $|s(x)|\le \|x\|$.
\item $s$ is continuous with respect to the norm.
\item $s$ extends to an algebra homomorphism $\Ahat(G) \to\bC$.
\end{enumerate}
\end{prop}
\begin{proof}
The implications (ii) $\Rightarrow$ (iii) $\Rightarrow$ (iv) are
clear, and the implication (iv) $\Rightarrow$ (i) follows from
Remark~\ref{rk:continuous}. So it remains to prove that
(i) $\Rightarrow$ (ii). Suppose that $s$ is core-bounded, and
write $x=\sum_ia_i[M_i]$ where the $M_i$ are
indecomposable. Then $s(x)=\sum_ia_is([M_i])$ and so
\begin{equation*}
|s(x)|\le \sum_i|a_i| \, |s([M_i])| \le \sum_i|a_i|\dim\core(M_i) =\|x\|.
\qedhere
\end{equation*}
\end{proof}

\begin{theorem}\label{th:species}
If $M$ is a $kG$-module then
\[ \npj_G(M)=\sup_{s\colon a(G)\to\bC} |s([M])| \]
where the supremum runs over the core-bounded species of
$a(G)$. Furthermore, there exists a core-bounded species $s$ of $a(G)$
such that $s([M])=\npj_G(M)$.
\end{theorem}

\begin{proof}
This equality follows from
Theorem~\ref{th:spec-radius}, Corollary~\ref{co:radius}
and Proposition~\ref{pr:core-bounded}. The final statement follows
from Theorem~\ref{th:SpecM}.
\end{proof}

\begin{cor}
For any $kG$-module $M$, the restriction of $\npj_G$ to the
sub-semiring of $A(G)$ consisting of elements of the form
$f([M])= \sum_{i=0}^n a_i[M]^i$ with the $a_i$ real and non-negative
is additive and multiplicative.
\end{cor}
\begin{proof}
By Theorem~\ref{th:species}, we may choose a species $s$
that maximises $|s([M])|$, and such that $s([M])=\npj_G(M)$.
Such an $s$ also maximises $|s(f([M]))|$.
\end{proof}

\begin{question}
What can be said about the \emph{quasi-nilpotent} elements of
$\Ahat(G)$, namely the elements $a$ satisfying
$\displaystyle\lim_{n\to\infty}\sqrt[n]{\|a^n\|}=0$?
These are the elements on which all core-bounded species vanish,
and they form the Jacobson radical of $\Ahat(G)$.
Is this the closure of the nil radical, or are there
more subtle ways of producing quasi-nilpotent elements?

Examples in analysis of
quasi-nilpotent operators which are not nilpotent
can be found in Examples~2.1.6 and~2.1.7 of Kaniuth~\cite{Kaniuth:2009a}.
Quasi-nilpotent elements also go by other names
in the literature. For example in \S I.4 of Gelfand, Raikov and
Shilov~\cite{Gelfand/Raikov/Shilov:1964a}  they are
called \emph{generalised nilpotent}, while in
Rickart~\cite{Rickart:1974a} they are called 
\emph{topologically nilpotent}. 
\end{question}

\begin{rk}\label{rk:Banach}
Many of the properties of $\npj_G$ that we have described correspond
to well-known properties of the spectral radius in a Banach algebra,
although we have chosen an exposition that is self-contained except
for Theorem~\ref{pr:Gelfand}. This applies to Theorem~\ref{th:omni}
(i), (ii), (iv), (viii), the second inequality of (x), (xi), (xiii)
and (xiv). Others require a Banach lattice: these can record the
special role played by the linear combinations of modules with
real non-negative coefficients, which roughly approximate the
image of genuine modules as opposed to virtual ones; see, for
example, the book by Schaefer~\cite{Schaefer:1974a}. The first
inequality of Theorem~\ref{th:omni} (x) and also (xii) correspond
to facts about Banach lattices, as does Theorem~\ref{th:SpecM}
(\cite[Prop.\ V 4.1]{Schaefer:1974a}).

If we knew that $\Ahat(G)$ was a \emph{symmetric} Banach algebra,
then Conjecture~\ref{conj:EndM} would follow. Symmetric Banach
algebras are extensively discussed in \S 4.7 of 
Rickart~\cite{Rickart:1974a} and in \S I.8 of
Gelfand, Raikov and Shilov~\cite{Gelfand/Raikov/Shilov:1964a}.
\end{rk}

\section{Cyclic groups}
\label{se:cyclic}

The computations in this section are based on Green \cite{Green:1962a}.
For the purpose of this section only, let $G=\bZ/p$ be the
cyclic group of order $p$, where $p>2$ is the characteristic
of the field $k$, and let $M_j$ be the indecomposable
$kG$-module of dimension $j$ for $1\le j\le p$. Then we have
\[ M_2 \otimes M_j \cong \begin{cases}
M_2 & j=1 \\
M_{j+1} \oplus M_{j-1} & 2\le j \le p-1 \\
M_p \oplus M_p & j=p. \end{cases} \]

Let $U_j(x)$ be the Chebyshev polynomial of the second kind, defined by
the recurrence relation $U_0(x)=1$, $U_1(x)=2x$,
$U_j(x)=2xU_{j-1}(x)-U_{j-2}(x)$ ($j\ge 2$). These polynomials 
satisfy
\[ U_j(\cos\theta)=\frac{\sin(j+1)\theta}{\sin\theta} \]
for all $j\ge 0$. For $j\ge 1$,
the roots of $U_j(x)$ are real and distinct, symmetric about
$x=0$, and given by
\[ x=\cos(m\pi/(j+1)),\quad 1\le m \le j. \]

We define $f_j(x)=U_{j-1}(x/2)$. So $f_1(x)=1$, $f_2(x)=x$, and
$xf_j(x)=f_{j+1}(x)+f_{j-1}(x)$ ($j\ge 2$). Then we have
\[ f_j\left(\frac{\sin 2\theta}{\sin\theta}\right) =
  \frac{\sin j\theta}{\sin\theta}. \]
Note that $f_p(x)$ is an irreducible polynomial in $x^2$. For example,
we have $f_3(x)=x^2-1$ and $f_5(x)=x^4-3x^2+1$. The roots of $f_p(x)$
are given by $x=2\cos(m\pi/p)$ $(1\le m\le p-1)$.

In the ring $a(G)/a(G,1)$, we have $[M_2][M_ j]=[M_{j+1}]+[M_{j-1}]$
and $[M_p]=0$. It follows that $f_j([M_2])=[M_j]$ and $f_p([M_2])=0$.
We therefore have
\[ a(G)/a(G,1) \cong \bZ[X]/(f_p(X)) \]
where $X$ corresponds to $[M_2]$. The core-bounded species of
$a(G)$ are just the non-Brauer species, and are
given by
\[ s_m\colon [M_2]\mapsto 2\cos(m\pi/p)\qquad (1\le m\le p-1). \]

\begin{theorem}\label{th:Z/p}
In the case of the indecomposable $kG$-module $M_j$ ($1\le j\le p-1$)
for the cyclic group $G$ of order $p$, the characteristic of $k$, we have
\[ \npj_G(M_j) = \frac{\sin (j\pi/p)}{\sin (\pi/p)} = f_j(\npj_G(M_2)) \]
where
\[ \npj_G(M_2)=\frac{\sin (2\pi/p)}{\sin (\pi/p)}=2\cos (\pi/p) . \]
In fact, $\npj_G$ is additive and multiplicative on modules,
so this determines $\npj_G$ on any module.
\end{theorem}
\begin{proof}
It follows from Theorem \ref{th:species} that for fixed $j$ we
need to maximise $|s_m([M_j])|$ over the core-bounded 
species $s_m$. By the discussion above,
\[ s_m([M_j])=s_m(f_j([M_2])=f_j(s_m([M_2]))=f_j(2 \cos (m \pi / p))
=U_{j-1}(\cos m \pi / p) = \frac{\sin (jm\pi/p)}{\sin (m\pi/p)}. \]
Express the sines in terms of $e^{im \pi/p}$, expand as a geometric
series and pair conjugate terms. If $j$ is odd, say $j=2r+1$, the
result is $1+ \sum_{s=1}^r \cos (sm \pi/p)$; the case when $j$
is even is similar and is left to the reader.

Thus we want to maximise the sum of the elements of a certain
class of $r$-element subsets of $\{ \cos (t \pi /p) : 1 \le t \le p-1\}$.
Clearly, the maximum over all $r$-element subsets is obtained
by choosing the $r$ largest elements, i.e., $t= 1, \ldots ,r$, which
in our case is attained with $m=1$.

Note that $m=1$ yields the maximum for all the $M_j$. Since
each $s_1([M_j])$ is a positive number, $s_1$ also yields the
maximum on all sums of the $M_j$, i.e., on all modules.
We have $\npj_G(M)= s_1([M])$ for all modules and the last part
of the theorem follows.
\end{proof}

\section{Methods of Calculation}\label{se:calc}

We recall some basic facts from Banach theory that we will use.

A linear operator (i.e., linear map) $T$ from a Banach space $B$
to itself is said to be bounded if
$\|T \|_{\op} := \sup \{ \|Tx \| : \|x \| =1 \}$ is finite.
The space of all such bounded operators forms a Banach space
$\mathcal B (B)$ with norm $\| \cdot \|_{\op}$. If $B$ is finite
dimensional then $\sigma(T)$ is just the finite set of eigenvalues,
so the set of roots of the characteristic polynomial of $T$, and
$\rho(T)$ is the largest of the absolute values of these.

If we start with a Banach algebra $A$, then any $a \in A$ yields
a bounded linear operator $T_a \in \mathcal B (A)$ by $T_ax=ax$
for $x \in A$. It is easy to check that $\| T_a \|_{\op} = \| a \|$;
we will usually omit the subscript $\op$ and often identify $T_a$
with $a$. The spectra might differ, but it follows from
Proposition~\ref{pr:Gelfand} that the spectral radii agree:
$\rho(T_a)=\rho(a)$.

We saw in Lemma~\ref{le:s-Omega} that for a core-bounded
species $s$ we have $s([\Omega k])= \lambda$ for some
complex number $\lambda$ with $|\lambda|=1$. It follows
that $s$ vanishes on the ideal of $\Ahat(G)$ generated by
$[\Omega k] - \lambda [k]$ and so factors through the quotient
Banach algebra by the closure of this ideal, which we denote by
$\Ahat(G)/(\Omega-\lambda)$. Thus we can find
$\npj_G(M)=\rho(T_M)$ by calculating it on each of these
quotients and taking the maximum value.

Because of the form of the definition of $\npj_G(M)$, we can
calculate it on any subalgebra of $\Ahat(G)$ that contains all
the tensor powers of $M$ and similarly for
$\Ahat(G)/(\Omega-\lambda)$. If we consider the operator
$T_M$, we can even restrict to any Banach subspace that
contains some tensor power $M^{\otimes n}$ of $M$ and that is
closed under tensor product with  $M^{\otimes n}$, by
Theorem~\ref{th:omni}\,(xiv).

Our strategy will be to  use these observations to reduce
to the finite dimensional case.

Of course, an operator $A$ on a finite dimensional vector
space with given basis can be represented by a square matrix
$(A_{i,j})$. There are various possible norms that we can
use on matrices. One is the operator norm induced from
a norm on the vector space. Another, which we will use
later, is $\| A \|_{\max} = \max_{i,j} \{ |A_{i,j}|\}$; this is a
Banach space norm, but it is not submultiplicative.
However, any two vector space norms on a finite dimensional
vector space, in this case the vector space of $n \times n$
matrices, are commensurate, i.e., there is a positive real
number $c$ such that $c^{-1} \| x \|_1 \le \| x \|_2 \le c \| x \|_1$
for all $x$. It follows that $\lim_{n \to \infty} \sqrt[n]{\| x \|}$
is independent of the norm, so yields $\rho(x)$, regardless
of whether the norm is submultiplicative or not.

Note that if the entries $A_{i,j}$ in the matrix for $A$ are
integers then $\rho(A)$ must be an algebraic integer.

The following standard lemma will be useful later, when
we look more closely at the quotient $\Ahat(G)/(\Omega-\lambda)$.
For a matrix $B$, we write $|B|$ for the matrix of absolute
values of the entries of $B$. If $A_1$ and
$A_2$ are matrices of the same size with real entries, we write $A_1\le A_2$
to indicate that each entry of $A_1$ is
less than or equal to the corresponding entry of $A_2$.

\begin{lemma}\label{le:Omega-eq-1}
Let $A$ be a square matrix with non-negative real entries,
and let $B$ be a complex matrix of the same size satisfying $|B|\le A$.
Then $\rho(B)\le \rho(A)$.
\end{lemma}
\begin{proof}
We have in general $|XY| \le |X| \, |Y|$ whenever the product 
is defined, and hence
$|B^n|\le A^n$. Thus $\| B^n \|_{\max} \le \| A^n \|_{\max}$.
Taking $n$th roots and then the limit as $n$ tends to infinity
yields the result.
\end{proof}

Finally, we formulate a result that depends heavily on the fact
that our norm is additive on modules; it could be generalised
to a Banach lattice with a norm that is additive on the positive cone.

\begin{prop}\label{pr:subop}
Suppose that we have a bounded operator $T$ on $\Ahat(G)$
that takes $kG$-modules to $kG$-modules and for some $m \in \mathbb N$
we have $kG$-modules $S_1, \ldots ,S_m$ and $Y_1, \ldots ,Y_m$ with
none of the $S_i$ projective. Suppose that there are non-negative
integers $A_{i,j}$ such that
\[ T[S_i] = \sum_j A_{i,j}[S_j] + [Y_i], \quad i = 1, \ldots , m \]
and consider the matrix $A = (A_{i,j})$. Then $\rho (T) \ge \rho (A)$.
\end{prop}
\begin{proof}
By induction, for any $n\ge1$ there are $kG$-modules $Z_{i,n}$ such that
\[ T^n[S_i] = \sum_j (A^n)_{i,j}[S_j] + [Z_{i,n}], \quad i = 1, \ldots, m. \]	
Note that this does not require the $[S_i]$ to be linearly
independent. Because the norm is additive on sums of modules
with non-negative coefficients,
we obtain
\[ \| T^n \|_{\op} \|[S_i]\| \ge \| T^n [S_i] \| =
\| \sum_j (A^n)_{i,j} [ S_j ] + [ Z_{i,n}] \| \ge
\max_j \{  (A^n)_{i,j} \}  \| [S_j] \| \ge \max_j \{  (A^n)_{i,j} \}, \]
since $S_j$ is not projective, so $\| [S_j] \| \ge 1$.	It follows that for some $i$ we have
\[ \| T^n \|_{\op}  \| [S_i] \|  \ge \max_{i,j} \{  (A^n)_{i,j} \}
= \| A^n \|_{\max}. \]
Taking $n$th roots and then the limit as $n$ tends to infinity
yields $\rho(T)\ge \rho (A)$.
\end{proof}

\section{Modules $M$ with $\npj_G(M)=\sqrt{2}$}\label{se:sqrt2}

In Theorem~\ref{th:sqrt2}, we showed that if $M$ is neither projective
nor endotrivial then we have $\npj_G(M)\ge\sqrt{2}$. In this section, we
investigate the case of equality.

\begin{lemma}\label{le:EN}
For any $kG$-module we have:
\begin{enumerate}
\item
If $M = E \oplus N$ with $E$ endotrivial and
$\npj_G(M \otimes M^*) < 4$ then $N$ is projective.
\item
If $M \otimes M^*$ is endotrivial then so is $M$.
\end{enumerate}
\end{lemma}
\begin{proof}
For part (i), we have $M \otimes M^* \cong k \oplus (E \otimes N^*)
\oplus (E^* \otimes N) \oplus (N \otimes N^*)$.
Theorem~\ref{th:ge-m} shows that $\npj_G(M \otimes M^*) \ge 4$
unless at least one of the terms in the sum is projective. If this is
the case then the tensor product of the modules in the sum is
projective. This tensor product is $N^{\otimes 2} \otimes
N^{*\,\otimes 2}$ plus projectives.
Lemmas~\ref{le:MM*} and~\ref{le:powerproj} show that if this is
projective then $N$ is projective.

For the second part, notice that if $M \otimes M^*$ is
endotrivial then $p$ cannot divide the dimension of $M$.
Thus $k$ is a summand of $M \otimes M^*$, and since
$M\otimes M^*$ is endotrivial,  the complementary summand
is projective.
\end{proof}

\begin{prop}\label{pr:1+s2}
If $1< \npj_G(M \otimes M^*) < 1+ \sqrt{2}$,
then $p$ divides the dimension of $M$.
\end{prop}

\begin{proof}
Suppose that the dimension of $M$ is not divisible
by $p$. By Theorem~\ref{th:E}, without loss of generality we may
assume that $G$ is elementary abelian.
By Lemma~\ref{le:K}, we may also assume that $k$ is algebraically closed.

We first show that we may suppose that $M$ is indecomposable.
Otherwise, choose an indecomposable summand	
$M_1$ of $M$ with dimension not divisible by $p$. Then
$\npj_G(M_1 \otimes M_1^*) < 1 + \sqrt{2}$, and we claim that
$1<\npj_G(M_1 \otimes M_1^*)$ so that we may replace $M$ by $M_1$.
If not, then $M_1 \otimes M_1^*$ is endotrivial by
Theorem~\ref{th:endotriv}, so $M_1$ is endotrivial by
Lemma~\ref{le:EN}\,(ii). Now apply Lemma~\ref{le:EN}\,(i)
to $M$ to see that $M = M_1 \oplus \proj$, so
$\npj_G(M_1 \otimes M_1^*) = \npj_G(M \otimes M^*) >1$, a
contradiction. Thus we may assume that $M$ is indecomposable.
	
We have $M \otimes M^* \cong k \oplus X$ with $X$ non-projective.
So using Theorem~\ref{th:k-plus-M} we have
\[ 1 + \sqrt{2}  > \npj_G(M\otimes M^*)=\npj_G(k\oplus X)=1+\npj_G(X). \]
By Theorem~\ref{th:sqrt2}, the only possibility is $\npj_G(X)=1$,
and by Theorem~\ref{th:endotriv},
$X$ is endotrivial.
By Proposition~\ref{pr:Dade} we have $X\cong \Omega^r k \oplus\proj$.
We know that $X$ is self dual, hence if $G$ is not cyclic we have
$r=0$ and so $M \otimes M^* \cong k \oplus k \oplus \proj$.
This contradicts Theorem~2.1 of
Benson and Carlson~\cite{Benson/Carlson:1986a} (this is where we
need to use the statements that $M$ is
indecomposable of dimension
not divisible by $p$, and $k$ is algebraically
closed). If $G$ is cyclic then $\Omega$ has period two and the
only other possibility is $r=1$. Thus $M \otimes M^* \cong k \oplus
\Omega k \oplus \proj$. This contradicts our assumption that $p$
does not divide the dimension of $M$, so the lemma is proved.
\end{proof}

The next proposition involves a number
$\alpha \approx 2.839286755\dots$, which is the unique real root of
the polynomial $x^3 -4x^2 +4x-2$.

\begin{prop}\label{pr:1.685}
If $M$ has dimension divisible by $p$ and $M \otimes M^* \otimes M \cong M
\oplus M \oplus X$ with $X$ not projective, then $\npj_G(M \otimes M^*)\ge \alpha$.
\end{prop}
\begin{proof}
The dimension of $X$ is divisible by $p$, so by
Proposition~\ref{pr:AC}\,(ii), $X \otimes X^* \otimes X
\cong X \oplus X \oplus Y$, but we have no information on whether $Y$
is projective. We have
\begin{align*}
(M\otimes M^*) \otimes (X \otimes M^*) &
\cong 2(X\otimes M^*) \oplus  (X \otimes X^*) \\
(M\otimes M^*) \otimes (X \otimes X^*) &
\cong (M\otimes M^* \otimes X \otimes X^*) \\
(M\otimes M^*) \otimes (M\otimes M^* \otimes X \otimes X^*) &
\cong 2(X \otimes M^*)\oplus 2(M \otimes M^* \otimes X \otimes X^*)
\oplus (Y\otimes M^*).
\end{align*}
Set $T=M\otimes M^*$, $S_1=X \otimes M^*$,
$S_2=X \otimes X^*$ and
$S_3=M\otimes M^* \otimes X \otimes X^*$.
None of the $S_i$ are projective.
Indeed, for $S_2$ this is proved in Lemma~\ref{le:MM*}.
Since $S_2$ is isomorphic to a summand of $S_1\otimes M \otimes M^*$, it follows
that $S_1$ is not projective.
Finally, $S_3 \cong S_1 \otimes S_1^*$, so $S_3$ is not
projective by Lemma~\ref{le:MM*} again.

We regard tensoring with $T$ as an operator and,
ignoring $Y\otimes M^*$ for the moment, the above isomorphism
says that the action of $T$ on the ordered set $\{S_1,S_2,S_3\}$
is recorded in the matrix
\[ A=\begin{pmatrix}
 2&1&0 \\ 0&0&1 \\ 2&0&2
\end{pmatrix} .\]
 The characteristic polynomial of the matrix $A$ is 
$x^3-4x^2+4x-2$, so
$\rho(A)= \alpha$.
Applying Proposition~\ref{pr:subop}, we obtain
\begin{equation*}
\npj_G(M \otimes M^*)=\rho(T)\ge \rho (A)= \alpha .
\qedhere
\end{equation*}
\end{proof}

\begin{rk}
The appeal to Proposition~\ref{pr:subop} at the end of the proof
of Proposition~\ref{pr:1.685} can be expressed more
na\"{\i}vely as follows. We have $T\otimes T \cong 2T \oplus
S_1$, $T \otimes S_1 \cong 2S_1 \oplus S_2$, $T\otimes S_2 \cong
S_3$, $T \otimes S_3 \cong 2S_1 \oplus 2S_2 \oplus (Y \otimes M^*)$.
So ignoring some summands,
we see that $T^{\otimes( n+2)}$ has
\[ \left(\begin{array}{ccc} 1&0&0 \end{array}\right)
A^n\left(\begin{array}{c}  S_1\\ S_2 \\ S_3\end{array}\right) \]
as a direct summand. It follows that the number of non-projective direct
summands of $T^{\otimes(n+2)}$ is at least the sum of the entries in
$\left(\begin{array}{ccc} 1&0&0 \end{array}\right)A^n$.
By Frobenius--Perron theory, this number is bounded below by
a constant multiple of $\alpha^n$, since
$\alpha$ is the largest real root of this matrix. So
we have $\npj_G(T)\ge \alpha$.
\end{rk}

\begin{theorem}\label{th:equiv-sqrt2}
For any non-projective $kG$-module $M$, the following conditions are equivalent.
\begin{enumerate}
\item
$1 < \npj_G(M \otimes M^*) < 1+\sqrt{2}$.
\item
$ \npj_G(M \otimes M^*) = 2$.
\item
$M \otimes M^* \otimes M \cong M
\oplus M \oplus \proj$.
\item
$[M \otimes M^*]^2  = 2 [M \otimes M^*]$ in $a(G)/a(G,1)$.
\end{enumerate}
\end{theorem}
\begin{proof}
We begin by showing that (i) implies (iii). 
	
If the dimension of $M$ is divisible by $p$ then by
Proposition~\ref{pr:AC} we have
\[ M \otimes M^* \otimes M \cong M \oplus M \oplus X. \]
If $X$ is not projective then by
Proposition~\ref{pr:1.685} we have $\npj_G(M \otimes M^*)\ge \alpha>1+\sqrt{2}$,
contradicting (i).
	
On the other hand, if the dimension of $M$ is not divisible
by $p$ then Proposition~\ref{pr:1+s2} shows that
$\npj_G(M \otimes M^*)$ cannot lie between $1$ and $1+ \sqrt{2}$,
also contradicting (i).

To see that (iii) implies (iv), tensor with $M^*$. Then it is
straightforward to see that (iv) implies (ii) and (ii) implies (i).
\end{proof}

\begin{cor}\label{co:MMM}
If a $kG$-module $M$ satisfies $1<\npj_G(M)<\sqrt{1+\sqrt{2}} \approx 1.553773974\dots$ then
it satisfies the conditions of Theorem~\ref{th:equiv-sqrt2}.
\end{cor}
\begin{proof}
The inequalities in the corollary imply those in 
Theorem~\ref{th:equiv-sqrt2}\,(i),
by Theorem~\ref{th:tensor} and Lemma~\ref{le:EN}\,(ii).
\end{proof}

\begin{rks}
\begin{enumerate}
\item
If Conjecture~\ref{conj:EndM} holds, then the converse
of Corollary~\ref{co:MMM} also holds, provided $M$ is not projective.
\item
If $M$ satisfies $M \otimes M^* \otimes M\cong
M \oplus M \oplus \proj$ then the restriction of $M$
to any cyclic subgroup of $G$ of order $p$ is projective. This can be
seen by noting that the module $M$ is self-dual on
restriction to any such cyclic subgroup $C$ of $G$,
and the relation $[M]^3=2[M]$ in $a(G)/a(G,1)$
yields that $\npj_C(M)$ is equal to $0$ or $\sqrt{2}$; the latter value can
be seen to be impossible from the calculations in
Section~\ref{se:cyclic}. It is not clear whether
there are any such modules except in the case that $p=2$ and a Sylow
$2$-subgroup of $G$ is isomorphic to $\bZ/2 \times \bZ/2$.
\end{enumerate}
\end{rks}

We end this section with a related conjecture, based on the
known examples of modules $M$ with $\npj_G(M)<2$. Such
modules are quite hard to construct.

\begin{conj}
If $M$ is a $kG$-module with $\npj_G(M)<2$ then
$\npj_G(M)=2\cos (\pi/m)$ for some integer $m\ge 2$.
\end{conj}

\begin{rk}
The case $m=2$ gives $\npj_G(M)=0$,  $m=3$ gives
$\npj_G(M)=1$, and the case $m\ge 4$ gives
$\npj_G(M)\ge\sqrt{2}$.
This fits well with Theorems~\ref{th:sqrt2},~\ref{th:endotriv} and~\ref{co:MMM}.
The two dimensional representation of $\bZ/p$ in characteristic $p$
is an example with $\npj_G(M)=2\cos(\pi/p)$, see Theorem~\ref{th:Z/p}.
It follows from the computations in Alperin~\cite{Alperin:1976b}
and Section~3 of Craven~\cite{Craven:2013a} that if
$M$ is the two dimensional natural module for $SL(2,q)$ with $q$
a prime power, then $M$ is algebraic and $\npj_G(M)=2\cos(\pi/q)$.
This can also be seen by using the weight theory of tilting modules.
\end{rk}

\begin{question}
What are the general properties of modules $M$ with $\npj_G(M)<2$?
\end{question}

\section{Eventually Recursive Functions}\label{se:recursive}

\begin{defn}
We say that a function $\phi\colon\bZ_{\ge 0} \to \bZ_{\ge 0}$ is
\emph{eventually recursive} of degree $d$ if there exists
a homogeneous linear recurrence relation with constant coefficients,
in other words a recurrence relation of the form
\[ \phi(n)+ a_{1}\phi(n-1)+\dots + a_{d}\phi(n-d) = 0 \]
with $a_i\in\bZ$, which is satisfied for all large enough integers
$n$. The recurrence relation of minimal degree eventually satisfied
by $\phi(n)$ is uniquely determined, and the corresponding
polynomial
\[ z^d + a_1z^{d-1} + \dots + a_d \]
has $a_d \ne 0$ and is called the \emph{characteristic polynomial} of the recurrence
relation.
\end{defn}

The following is a standard theorem from the theory of recurrence relations.

\begin{theorem}\label{th:er}
If $\phi$ is eventually recursive and $r$ is the radius of convergence
of the generating function $f(t)=\sum_{n=0}^\infty \phi(n)t^n$
then $1/r$ is an algebraic integer. It is the largest positive real root of the
characteristic polynomial.
\end{theorem}
\begin{proof}
It is easy to check that $(1+a_1t + \dots + a_dt^d)f(t)$
is a polynomial in $t$, and therefore $f(t)$ is a rational
function with denominator equal to $1+a_1t+\dots+a_dt^d$.
The poles of this rational function are the $1/\lambda_i$ where
the roots of the characteristic polynomial are $\lambda_i$.
The radius of convergence $r$ is therefore the smallest of the
$1/|\lambda_i|$.
Now apply Pringsheim's Theorem~\ref{th:Pringsheim} to $f(t)$ to see
that $r$ is a root of the characteristic polynomial.
\end{proof}

Based on a large number of computations using \textsf{Magma},
some of which are presented in Section~\ref{se:eg},
we make the following conjecture.

\begin{conj}\label{conj:er}
If $M$ is a $kG$-module then the function $\phi(n)=\cc_n^G(M)$
is eventually recursive.
\end{conj}

\begin{rks}\label{rks:recur}
\begin{enumerate}
\item
If $M$ is an algebraic module then Conjecture~\ref{conj:er} holds
for $M$.
\item
If Conjecture~\ref{conj:er}
holds for a module $M$, then by Theorem \ref{th:er},
$\npj_G(M)$ is an algebraic integer.
\item
In Example \ref{eg:3x3a},
the minimal equation for $M$ modulo projectives is $x^2(x-2)^3=0$.
The largest solution to this has $|x|=2$, and so $1/r=\npj_G(M)=2$.
The associated recurrence relation is $\cc_n^G(M)-8c_{n-3}^G(M)=0$ for
$n\ge 5$.
\item
In Example \ref{eg:3x3b},
the module $M$ is not algebraic, but nonetheless, the relation
\[ M\otimes M \cong M^* \oplus \Omega(M^*) \]
can be used to produce a recurrence relation for $\cc_n^G(M)$ for
large enough $n$.
\end{enumerate}
\end{rks}

The consequence of Conjecture~\ref{conj:er} that 
$\npj_G(M)$ is an algebraic integer is at least 
consistent with the last result of this section.

\begin{prop}\label{prop:countable} 
The invariant $\npj_G(M)$ can take only countably 
many values (for all primes, fields, finite groups and 
finitely generated modules).
\end{prop}

\begin{proof}
We know that $\npj_G$ does not vary with field extension, 
by Lemma~\ref{le:K}, and any representation over any field 
can be realised over a finitely generated field---the subfield 
generated over $\mathbb F_p$ by the entries of the matrices 
in some matrix form of the representation---so we only 
need to consider finitely generated fields. A finitely generated 
field is countable and there are only countably many such fields. 
To see this, add the generators one by one. If a generator is 
algebraic over the field produced at the previous stage then 
there are only countably many possibilities for its minimal 
polynomial and hence for the new field. If it is transcendental, 
the new field is already determined. Thus, for a given group, 
field and dimension, there are only countably many 
representations (consider matrices again) and so only countably 
many possible values of $\npj_G$. There are also only countably 
many possibilities for group, field and dimension.
\end{proof}

This proof also works for any other invariant of $kG$-modules 
that is preserved under field extension.

\section{Omega-Algebraic Modules}\label{se:Omega}

\begin{defn}
A $kG$-module $M$ is called \emph{Omega-algebraic} if the
non-projective indecomposable direct summands
of the modules $M^{\otimes n}$ fall into finitely many orbits of the
syzygy functor $\Omega$.
\end{defn}

An example of an Omega-algebraic module may be found in
Example \ref{eg:3x3b}.
We'll see more examples in the next section, as well as evidence
that not all modules are Omega-algebraic.
A weaker form of Conjecture~\ref{conj:er} is the following.

\begin{conj}
If $M$ is an Omega-algebraic $kG$-module then the 
function $\phi(n)=\cc_n^G(M)$ is eventually recursive.
\end{conj}

We can calculate $\npj_G(M)$ for an Omega-algebraic module
$M$ as follows, using the methods of Section~\ref{se:calc}.
First we restrict to the subspace $V$ of $\Ahat(G)$ generated by
all the indecomposable summands of all the tensor powers
of $M$, together with all their syzygies.
Choose representatives of the $\Omega$-orbits of these
indecomposable summands, say $M_1,\dots,M_d$; by
hypothesis there are only finitely many. Each $M\otimes M_i$ decomposes as
a direct sum of modules of the form $\Omega^m(M_j)$. 
The operation of tensoring with $M$ gives us a
$d\times d$ matrix $X(\Omega)$,
whose entries are Laurent polynomials in the operator $\Omega$
which have non-negative coefficients.

Now, for $\lambda\in\bC$
with $|\lambda|=1$, form the quotient $V/(\Omega-\lambda)$ of
$V$ by the linear span of the elements 
$[\Omega^{m+1}(M_j)]-\lambda[\Omega^m(M_j)]$
with $m\in\bZ$, $1\le j \le d$. This quotient is
a finite dimensional vector space with
basis the images of the $[M_j]$; the matrix corresponding
to tensoring with $M$ is $X( \lambda)$, meaning that we
substitute $\lambda$ for $\Omega$. Since the Laurent
polynomials in $\Omega$ have non-negative coefficients
we have $|X(\lambda)| \le X(1)$, so we can apply
Lemma~\ref{le:Omega-eq-1} to see that
$\rho(X(\lambda)) \le \rho(X(1))$. It follows, using the
discussion at the beginning of Section~\ref{se:calc}, 
that $\npj_G(M)$ is the largest eigenvector of the 
matrix $X(1)$. In particular, it is an algebraic integer.

\section{Some examples}\label{se:eg}

Example~\ref{eg:3x3b} is an example of an Omega-algebraic module
which is not algebraic. Here are some more complicated examples.
The computations use the methods outlined in the previous section and in Section~\ref{se:calc}.

\begin{eg}\label{eg:Omega-alg3}
Let $G=\langle g,h\rangle \cong\bZ/3\times\bZ/3$ and $k=\bF_3$.
Let $M_6$ be the six-dimensional $kG$-module given by the following
matrices, which has the diagram shown:
\[ g\mapsto\begin{pmatrix}
1&0&1&0&0&0\\ 0&1&0&0&0&0\\ 0&0&1&0&1&0\\
0&0&0&1&0&1\\ 0&0&0&0&1&0\\ 0&0&0&0&0&1
\end{pmatrix}\qquad
h\mapsto\begin{pmatrix}
1&0&0&1&0&0\\ 0&1&0&0&1&0\\ 0&0&1&0&0&1\\
0&0&0&1&0&0\\ 0&0&0&0&1&0\\ 0&0&0&0&0&1
\end{pmatrix}\qquad
\vcenter{\xymatrix@=4mm{&\bul\ar@{-}[dl]\ar@{-}[dr]&&
\bul\ar@{-}[dl]\ar@{-}[dr] \\
\bul\ar@{-}[dr]&&\bul\ar@{-}[dl]&&\bul\\&\bul}} \]
For an explanation of the diagrams in this section, see Example~\ref{eg:3x3b}.
Then $M_6$ is Omega-algebraic. Letting
$M_3=k_{\langle g\rangle}{\uparrow^G}$ and $P$ be the projective
indecomposable module, we have
\begin{align*}
M_6 \otimes M_6 &\cong \Omega(M_6) \oplus \Omega^{-1}({M_6}^*)
\oplus M_3    \oplus P  \\
M_6 \otimes {M_6}^* &\cong
\Omega^{-1}(M_6) \oplus \Omega({M_6}^*) \oplus M_3 \oplus P \\
M_6 \otimes M_3 &\cong \Omega(M_3) \oplus \Omega(M_3) \oplus
\Omega(M_3).
\end{align*}
The rows of the table in Figure~\ref{fig:Omega-alg3}
give the result of tensoring a module in the first column with $M_6$,
and writing it as a direct sum of the syzygies of the modules listed
in the first row of the table.

\begin{figure}[htb]
\[ \begin{array}{l|cccc}
&M_6&{M_6}^*&M_3&P \\ \hline
M_6& \Omega & \Omega^{-1} & 1 & 1 \\
{M_6}^*&\Omega^{-1} & \Omega & 1 & 1 \\
M_3 & 0 & 0 & 3\Omega & 0 \\
P & 0 & 0& 0 & 6
\end{array} \]
\caption{Table for Example~\ref{eg:Omega-alg3}}\label{fig:Omega-alg3}
\end{figure}

By the argument described at the end of Section~\ref{se:Omega},
we delete the projectives in this table and replace $\Omega$ by $1$, to
obtain the matrix
\[ \begin{pmatrix} 1&1&1 \\ 1&1&1\\ 0&0&3 \end{pmatrix} \]
Then we find the
eigenvalues of this matrix, which are $3$, $2$ and $0$. The largest of these is $3$,
so $\npj_G(M_6)=3$.
\end{eg}

\begin{eg}\label{eg:Omega-alg2}
Let $G=\langle g,h\rangle\cong\bZ/3 \times \bZ/3$ and $k=\bF_3$. This
time, let $M_6$ be the
self-dual six dimensional module given by the following matrices,
which has the diagram shown:
\[ g\mapsto \begin{pmatrix}
1&1&0&0&0&0\\ 0&1&0&0&0&0\\ 0&0&1&1&0&0\\
0&0&0&1&0&0\\ 0&0&0&0&1&1\\ 0&0&0&0&0&1
\end{pmatrix}\qquad
h\mapsto \begin{pmatrix}
1&0&0&1&0&0\\ 0&1&0&0&0&0\\ 0&0&1&0&0&1\\
0&0&0&1&0&0\\ 0&0&0&0&1&0\\ 0&0&0&0&0&1
\end{pmatrix} \qquad
\vcenter{\xymatrix@=4mm{\bul \ar@{-}[dr]&&\bul\ar@{-}[dl]\ar@{-}[dr]
&&\bul\ar@{-}[dl]\ar@{-}[dr]\\&\bul&&\bul&&\bul}} \]
Then $M_6$ is Omega-algebraic. There are modules $M_{15}$, $M_{21}$,
$M_{27}$, ${M_{27}}^*$, $M_{39}$, ${M_{39}}^*$ and the
projective indecomposable module $P$, such that the rows of the table in
Figure~\ref{fig:Omega-alg2} give the result of tensoring with $M_6$.

\begin{figure}[htb]
{\tiny
\[ \renewcommand{\arraystretch}{1.4}
\begin{array}{l|cccccccccccccccc}
&M_6 & M_{15} & M_{21}&M_{27}&{M_{27}}^*&M_{39}&{M_{39}}^*&M_{66}& P \\ \hline
M_6&&1&1\\
M_{15}&2&&&&&1&1\\
M_{21}&1&&&1&1&&&1\\
M_{27}&&&&2&2&&&&6 \\
{M_{27}}^*&&&&2&2&&&&6 \\
M_{39}&\Omega&&1&1&1&&2\Omega&&6 \\
{M_{39}}^* &\Omega^{-1}&&1&1&1&2\Omega^{-1}&&&6 \\
M_{66} &&&2&2&2&\Omega^{-1}&\Omega&&18 \\
P&&&&&&&&&6
\end{array} \]
}
\caption{Table for Example~\ref{eg:Omega-alg2}} \label{fig:Omega-alg2}
\end{figure}

If we throw out the projectives, and replace $\Omega$ by $1$,
the largest real eigenvalue of the remaining matrix will give the
value of $\gamma_G(M)$. In this example, the eigenvalues
are $4$, $3$, $1$, $0$, $0$, $0$, $-2$, $-2$, so $\npj_G(M)=4$.
\end{eg}

\begin{eg}\label{eg:Omega-alg}
Let $G=\bZ/5 \times \bZ/5$, and $k=\bF_5$. Let $M_3$ be the
three dimensional
module $kG/\Rad^2(kG)$. Then $M_3$ is Omega-algebraic. More precisely,
there are modules $M_6$, $M_8$, $M_{10}$, $M_{15}$, $M_{30}$,
$M_{35}$, $M_{45}$ and $M_{65}$ such that $M_8$, $M_{35}$ and $M_{65}$
are self-dual, and the rows of the table in Figure~\ref{fig:Omega-alg}
give the effect of tensoring with $M_3$.

\begin{figure}[htb]
{\tiny
\[ \renewcommand{\arraystretch}{1.4}
\begin{array}{l|cccccccccccccccc}
&k&M_3&{M_3}^*&M_6&{M_6}^*&M_8&M_{10}&{M_{10}}^*&M_{15}&{M_{15}}^*&
M_{30}&{M_{30}}^*&M_{35}&M_{45}&{M_{45}}^*&M_{65}\\ \hline
\ k         &&1\\
M_3        &&&1&1\\
{M_3}^*   &1&&&&&1\\
M_6        &&&&&&1&1\\
{M_6}^*   &&&1&&&&&&&1\\
M_8        &&1&&&1&&&&1\\
M_{10}     &&&&&&&&\Omega^{-1}&1\\
{M_{10}}^*&&&&&&&&&&&&1\\
M_{15}     &&&&&&&&&&1&1\\
{M_{15}}^*&&&&&&&&1&&&&&1\\
M_{30}     &&&&&&&2&&&\Omega^{-1}&&&1\\
{M_{30}}^*&&&&&&&\Omega&&&2&&&&&1\\
M_{35}     &&&&&&&&&2&&&1&&1\\
M_{45}     &&&&&&&&&\Omega^{-1}&&1&&&&1\\
{M_{45}}^*&&&&&&&&&\Omega^{-1}&&&&1&&&1\\
M_{65}     &&&&&&&\Omega^{-1}&&\Omega&&&&&1
\end{array} \]
}
\caption{Table for Example~\ref{eg:Omega-alg}} \label{fig:Omega-alg}
\end{figure}

Replacing $\Omega$ by $1$, the characteristic polynomial is
\[ (x^2-3x+1)(x^6-4x^3-1)
(x^8+3x^7+8x^6+6x^5+5x^4-x^3+12x^2+7x+9). \]
So writing $\tau$ for $(1+\sqrt{5})/2$, the eigenvalues are $1+\tau$ and $2-\tau$;
$\tau$ and $1-\tau$
times cube roots of unity;
and the roots of a
degree eight irreducible with no real roots. The largest real
eigenvalue is $1+\tau\approx 2.618$, and so $\npj_G(M)=1+\tau$.
This happens to be the same as $\npj_H(M)$ for each cyclic subgroup
$H$ of $G$ in this case. This is because the restriction is $k\oplus
M_2$, so $\npj_H(M)$ can be computed using Theorems~\ref{th:k-plus-M}
and~\ref{th:Z/p}.

\end{eg}\pagebreak[3]

The next example shows how we can make use of
examples already computed.

\begin{eg}
Let $G=\bZ/3\times \bZ/3=\langle g,h\rangle$ and $k=\bF_3$.
Let $M_5$ be the five dimensional $kG$-module given by the following
matrices, which has the diagram shown:
\[ g\mapsto \begin{pmatrix}
	1&0&0&0&0\\ 0&1&1&0&0\\
	0&0&1&0&0\\ 0&0&0&1&1\\ 0&0&0&0&1
	\end{pmatrix}\qquad
	h\mapsto \begin{pmatrix}
	1&0&1&0&0\\ 0&1&0&0&1\\
	0&0&1&0&0\\ 0&0&0&1&0\\ 0&0&0&0&1
	\end{pmatrix} \qquad
	\vcenter{\xymatrix@=4mm{&\bul\ar@{-}[dl]\ar@{-}[dr]&&
	\bul\ar@{-}[dl]\ar@{-}[dr]
	\\\bul&&\bul&&\bul}}  \]
Consider the restriction of $M_5$ to the cyclic subgroup 
generated by $g$. This is a sum of the trivial module and 
two copies of the indecomposable module of dimension 2. 
By Theorem~\ref{th:Z/p} we have 
$\npj_{\langle g \rangle}(M_5)=1 + 4 \cos (\pi/3) = 3$ 
and by Theorem~\ref{le:subgroup} $\npj_G(M_5) \ge 3$.
	
There is a short exact sequence
\[ 0 \to M_5 \to \Omega^{-1}(M_3) \to k \to 0, \]
where $M_3$ is the three dimensional module
of Example~\ref{eg:3x3b}. Since $\npj_G(M_3)=2$,
by Theorem~\ref{th:Omega} we have $\npj_G(\Omega^{-1}(M_3))=2$.
So using Corollary~\ref{co:ses} we have $\npj_G(M_5)\le 3$.
Combining these bounds, we have $\npj_G(M_5)=3$. This example is quite
difficult to compute directly.
\end{eg}

Not all modules are Omega-algebraic. For example, 
the module of Example~\ref{eg:3x3b} inflated to $(\bZ/3)^3$,
with one of the factors acting trivially, is not Omega-algebraic. It is less clear whether
faithful modules can fail to be Omega-algebraic, but some evidence is
provided by the following example.

\begin{eg}\label{eg:non-Omega-alg}
Calculations using \textsf{Magma} give the following.
Let $M$ be the self-dual four dimensional module for
$G=\langle g,h\rangle\cong\bZ/3 \times \bZ/3$ over 
$\mathbb F_3$ given by the following matrices,
which has the diagram shown:
\[ g \mapsto \begin{pmatrix} 1&1&0&0\\0&1&0&0\\0&0&1&1\\0&0&0&1\end{pmatrix}
\qquad
h \mapsto \begin{pmatrix}  1&0&0&1\\0&1&0&0\\0&0&1&0\\0&0&0&1\end{pmatrix}
\qquad
\vcenter{\xymatrix@=4mm{\bul \ar@{-}[dr]&&\bul\ar@{-}[dl]\ar@{-}[dr]\\&\bul&&\bul}} \]
The beginning of the pattern is as follows:\pagebreak[3] 
\begin{align*}
M\otimes M\ &\cong k \oplus V_2 \oplus W &
M \otimes X\ & \cong 2X \oplus X^* \oplus \proj \\
M\otimes V_2\, &\cong M \oplus V_3 &
M \otimes X^* & \cong 2X^* \oplus X \oplus \proj \\
M \otimes W\, &\cong M \oplus X \oplus X^*
\end{align*}
Here, $V_2$, $W$ and $V_3$ are self-dual of dimensions 5, 10 and 16, 
$X$ has dimension 18, and its dual satisfies
$X^* \cong \Omega(X)\cong\Omega^{-1}X$.

Thereafter, at least as far as we have calculated,
there are self-dual indecomposable modules
denoted $V_i$ for all $i \ge 1$, where $V_1 = M$ 
and $V_2$ and $V_3$ are the same modules as above. 
The dimension of $V_i$ is $3i+a(i)$, where $a(i)$ 
depends only on the residue class of $i$ modulo 6.
\[
\begin{array}{c|cccccc}
i \pmod{6} & 0 & 1 & 2 & 3 & 4 & 5 \\
\hline
a(i) & 11 & 1 & -1 & 7 & 2 & 7
\end{array}
\]
The tensor product of $M$ with $V_i$ for $i\ge 2$ appears to be given by
\[ M \otimes V_i \cong \begin{cases}
V_{i-1} \oplus V_{i+1} \oplus \proj & \text{if } 3\nmid i \\
V_{i-1} \oplus V_{i+1} \oplus X \oplus X^* \oplus \proj & \text{if } 3\mid i.
\end{cases} \]
We have verified this pattern for $2\le i\le 35$.
Notice that $[X \oplus X^*]$ is an eigenvector of multiplication
by $[M]$ with eigenvalue 3;
this does not depend on having correctly spotted the pattern. 
Thus $\npj_G(M) \ge 3$.

Now consider the diagram that describes $M$.
If we remove the right  hand vertex we obtain the diagram
for the three dimensional module of Example~\ref{eg:3x3b},
which we will denote by $M_3$.
We know that $\npj_G(M_3) = 2$ and there is a short exact sequence
\[ 0 \to k \to M \to M_3 \to 0. \]
By Corollary~\ref{co:ses} we have $\npj_G(M) \le \npj_G(M_3) +
\npj_G(k) =3$.
Thus $\npj_G(M) \le 3$. Combining this with the previous bound gives
a proof that $\npj_G(M)=3$.

It appears that this example is not Omega-algebraic, but it still seems to satisfy
Conjecture~\ref{conj:er}, although we have not been able to write down
a proof of this.
\end{eg}

We end with a conjecture which is closely related to Conjecture~E of
Craven~\cite{Craven:2011a}, and which is motivated by extensive
computations using \textsf{Magma}.

\begin{conj}
If $M$ is an absolutely indecomposable module for $\bZ/p\times \bZ/p$
and the dimension of $M$ is divisible by $p$ then $M$ is
Omega-algebraic.
\end{conj}

\newcommand{\noopsort}[1]{}
\providecommand{\bysame}{\leavevmode\hbox to3em{\hrulefill}\thinspace}
\providecommand{\MR}{\relax\ifhmode\unskip\space\fi MR }
\providecommand{\MRhref}[2]{%
  \href{http://www.ams.org/mathscinet-getitem?mr=#1}{#2}
}
\providecommand{\href}[2]{#2}


\begin{thebibliography}{10}

\bibitem{Alperin:1976b}
J.~L. Alperin, \emph{{On modules for the linear fractional groups}}, Finite
  Groups, Sapporo and Kyoto, 1974 (N.~Iwahori, ed.), Japan Society for the
  Promotion of Science, 1976.

\bibitem{Alperin/Evens:1981a}
J.~L. Alperin and L.~Evens, \emph{{Representations, resolutions, and Quillen's
  dimension theorem}}, J.~Pure \& Applied Algebra \textbf{22} (1981), 1--9.

\bibitem{Auslander/Carlson:1986a}
M.~Auslander and J.~F. Carlson, \emph{{Almost-split sequences and group
  rings}}, J.~Algebra \textbf{103} (1986), 122--140.

\bibitem{Benson:1991b}
D.~J. Benson, \emph{{Representations and cohomology II: Cohomology of groups
  and modules}}, Cambridge Studies in Advanced Mathematics, vol.~31, Cambridge
  University Press, 1991, reprinted in paperback, 1998.

\bibitem{Benson/Carlson:1986a}
D.~J. Benson and J.~F. Carlson, \emph{{Nilpotent elements in the Green ring}},
  J.~Algebra \textbf{104} (1986), 329--350.

\bibitem{Benson/Parker:1984a}
D.~J. Benson and R.~A. Parker, \emph{{The Green ring of a finite group}},
  J.~Algebra \textbf{87} (1984), 290--331.

\bibitem{Bosma/Cannon:1996a}
W.~Bosma and J.~Cannon, \emph{{Handbook of Magma Functions}}, Magma Computer
  Algebra, Sydney, 1996.

\bibitem{Bryant/Kovacs:1972a}
R.~M. Bryant and L.~Kov\'acs, \emph{{Tensor products of representations of
  finite groups}}, Bull.\ London Math.\ Soc. \textbf{4} (1972), 133--135.

\bibitem{Carlson:1981a}
J.~F. Carlson, \emph{{Complexity and Krull dimension}}, Representations of
  Algebras, Puebla, Mexico, 1980, Lecture Notes in Mathematics, vol. 903,
  Springer-Verlag, Ber\-lin/New York, 1981, pp.~62--67.

\bibitem{Carlson:1981c}
\bysame, \emph{{Dimensions of modules and their restrictions over modular group
  algebras}}, J.~Algebra \textbf{69} (1981), 95--104.

\bibitem{Carlson:1996a}
\bysame, \emph{{Modules and Group Algebras}}, Lectures in Mathematics, ETH
  Z\"urich, Birkh\"auser Verlag, Basel, 1996.

\bibitem{Conway:1973a}
J.~B. Conway, \emph{{Functions of one complex variable}}, Graduate Texts in
  Mathematics, vol.~11, Springer-Verlag, Ber\-lin/New York, 1973.

\bibitem{Craven:2007a}
D.~A. Craven, \emph{{Algebraic modules for finite groups}}, {Ph.\ D.\
  Dissertation}, University of Oxford, 2007.

\bibitem{Craven:2011a}
\bysame, \emph{{Algebraic modules and the Auslander--Reiten quiver}}, J.~Pure
  \& Applied Algebra \textbf{215} (2011), no.~3, 221--231.

\bibitem{Craven:2013a}
\bysame, \emph{{On tensor products of simple modules for simple groups}},
  Algebras and Representation Theory \textbf{16} (2013), 377--404.

\bibitem{Dade:1978a}
E.~C. Dade, \emph{{Endo-permutation modules over $p$-groups, I}}, Ann.\ of
  Math. \textbf{107} (1978), 459--494.

\bibitem{Dade:1978b}
\bysame, \emph{{Endo-permutation modules over $p$-groups, II}}, Ann.\ of Math.
  \textbf{108} (1978), 317--346.

\bibitem{Etingof/Gelaki/Nikshych/Ostrik:2015a}
P.~Etingof, S.~Gelaki, D.~Nikshych, and V.~Ostrik, \emph{{Tensor categories}},
  Math.\ Surveys and Monographs, vol. 205, American Math.\ Society, 2015.

\bibitem{Feit:1982a}
W.~Feit, \emph{{The representation theory of finite groups}}, North Holland,
  Amsterdam, 1982.

\bibitem{Fekete:1923a}
M.~Fekete, \emph{{\"Uber die Verteilung der Wurzeln bei gewissen algebraische
  Gleichungen mit ganzzahligen Koeffizienten}}, Math.\ Zeit. \textbf{17}
  (1923), 228--249.

\bibitem{Gelfand:1941a}
I.~M. Gelfand, \emph{{Normierte Ringe}}, Rec.\ Math.\ (Mat.\ Sbornik) N.\ S.
  \textbf{51} (1941), no.~9, 3--24.

\bibitem{Gelfand/Raikov/Shilov:1964a}
I.~M. Gelfand, D.~Raikov, and G.~Shilov, \emph{{Commutative normed rings}},
  Chelsea, 1964.

\bibitem{Green:1962a}
J.~A. Green, \emph{{The modular representation algebra of a finite group}},
  Illinois J.\ Math. \textbf{6} (1962), 607--619.

\bibitem{Kaniuth:2009a}
E.~Kaniuth, \emph{{A course in commutative Banach algebras}}, Graduate Texts in
  Mathematics, vol. 246, Springer-Verlag, Ber\-lin/New York, 2009.

\bibitem{Kroll:1984a}
O.~Kroll, \emph{{Complexity and elementary abelian $p$-groups}}, J.~Algebra
  \textbf{88} (1984), 155--172.

\bibitem{Lax:2002a}
P.~D. Lax, \emph{{Functional analysis}}, John Wiley \& Sons Inc., New York,
  2002.

\bibitem{Rickart:1974a}
C.~E. Rickart, \emph{{General theory of Banach algebras}}, Van Nostrand
  Reinhold, 1974.

\bibitem{Schaefer:1974a}
H.~H. Schaefer, \emph{{Banach lattices and positive operators}}, Grundlehren
  der mathematischen Wissenschaften, vol. 215, Springer-Verlag, Ber\-lin/New
  York, 1974.

\bibitem{Titchmarsh:1939a}
E.~C. Titchmarsh, \emph{{The theory of functions}}, second ed., Oxford
  University Press, 1939.

\bibitem{Zemanek:1971a}
J.~R. Zemanek, \emph{{Nilpotent elements in representation rings}}, J.~Algebra
  \textbf{19} (1971), 453--469.

\bibitem{Zemanek:1973a}
\bysame, \emph{{Nilpotent elements in representation rings over fields of
  characteristic $2$}}, J.~Algebra \textbf{25} (1973), 534--553.

\end{thebibliography}
\end{document}